\documentclass{article}

\PassOptionsToPackage{numbers, compress}{natbib}
\usepackage{amsmath,amssymb,amsthm,fullpage,mathrsfs,pgf,tikz,caption,subcaption,mathtools,cancel}
\usepackage[ruled,vlined,linesnumbered]{algorithm2e}
\usepackage{natbib}
\usepackage{amsmath,amssymb,amsthm,mathtools}
\usepackage[utf8]{inputenc} 
\usepackage[T1]{fontenc}    
\usepackage{hyperref}       
\usepackage{url}            
\usepackage{booktabs}       
\usepackage{amsfonts}       
\usepackage{nicefrac}       
\usepackage{microtype}      
\usepackage[margin=1in]{geometry}
\usepackage{bbm}
\usepackage{enumitem}
\usepackage{xcolor}
\usepackage{cleveref}

\let\oldnl\nl
\newcommand{\nonl}{\renewcommand{\nl}{\let\nl\oldnl}}

\usepackage{aliascnt}

\newtheorem{theorem}{Theorem}[section]

\newaliascnt{corollary}{theorem}
\newtheorem{corollary}[corollary]{Corollary}
\aliascntresetthe{corollary}

\newaliascnt{proposition}{theorem}
\newtheorem{proposition}[proposition]{Proposition}
\aliascntresetthe{proposition}

\newaliascnt{lemma}{theorem}
\newtheorem{lemma}[lemma]{Lemma}
\aliascntresetthe{lemma}

\newaliascnt{definition}{theorem}
\newtheorem{definition}[definition]{Definition}
\aliascntresetthe{definition}

\newaliascnt{conjecture}{theorem}

\aliascntresetthe{conjecture}

\newaliascnt{question}{theorem}

\aliascntresetthe{question}

\newaliascnt{example}{theorem}

\aliascntresetthe{example}

\newaliascnt{remark}{theorem}

\aliascntresetthe{remark}

\newaliascnt{observation}{theorem}

\aliascntresetthe{observation}

\newaliascnt{claim}{theorem}
\newtheorem{claim}[claim]{Claim}
\aliascntresetthe{claim}

\newcommand{\R}{\mathbb{R}}

\DeclareMathOperator{\poly}{poly}

\newcommand{\norm}[1]{\| {#1} \|}

\newcommand*\circled[1]{\tikz[baseline=(char.base)]{
            \node[shape=circle,draw,inner sep=2pt] (char) {#1};}}

\title{Toward a KKL Theorem for any HDX}

\author{
Max Hopkins\thanks{Institute for Advanced Study, Princeton. nmhopkin@ias.edu. Supported by NSF Award DMS-2424441}
}

\begin{document}

\maketitle

\begin{abstract}

The KKL Theorem, a seminal result in boolean function analysis, characterizes the structure of low-influence (non-expanding) functions on the hypercube. While recent years have seen breakthrough results across a variety of areas relying on analogs of the KKL Theorem \textit{beyond the cube} (e.g.\ on product spaces, Grassmann graphs), further progress has been inhibited by our poor understanding of the phenomenon across more general domains. Motivated in this context, Bafna, Hopkins, Kaufman, and Lovett (STOC 2022) and Gur, Lifshitz, and Liu (STOC 2022) proved a generalized KKL-type Theorem for spectral high dimensional expanders (HDX). Their results, however, remain highly restricted due to strong quantitative expansion requirements on the underlying complex.

In this work, we introduce a simple local-to-global method for analyzing low influence functions on simplicial complexes. Using this method we prove a local-to-global KKL-type Theorem: any simplicial complex whose links satisfy a KKL-Theorem also satisfies such a result globally. Building on Gotlib and Kaufman (RANDOM 2023), we also prove a weaker dimension-dependent KKL-type Theorem for simplicial complexes with \emph{any non-trivial (two-sided) expansion}. As concrete applications of our framework, we give the first characterization of non-expanding functions on `combinatorial' HDX such as dense clique complexes and a corresponding Kruskal-Katona Theorem, as well as a small-set expansion theorem for the Ramanujan Complexes of Lubotzky, Samuels, and Vishne (EJC '05).
\end{abstract}
\section{Introduction}
In 1988, Kahn, Kalai, and Linial \cite{kahn1988influence} proved the KKL Theorem: a characterization of low influence (non-expanding) functions\footnote{Informally, a low-influence function on the hypercube is one which has low expected probability of changing value under a random coordinate flip.} on the hypercube that revolutionized the field of boolean function analysis.
In recent years, variants of the KKL Theorem \textit{beyond the cube}, including on products \cite{bourgain1992influence,talagrand1994russo,friedgut1996every,friedgut1998boolean,hatami2012structure,keevash2019hypercontractivity,lifshitz2019noise}, Lie groups \cite{ellis2024product}, and the Grassmann \cite{subhash2018pseudorandom,ellis2023analogue} have become a powerful tool across a surprisingly broad array of areas, leading, e.g., to the theory of sharp thresholds \cite{friedgut1999sharp}, multiple breakthroughs in PCP theory \cite{subhash2018pseudorandom,khot2017independent,dinur2018towards,dinur2018non,barak2018small,khot2018small,minzer2024near}, new streaming lower bounds \cite{fei2025multi}, and to progress on longstanding problems in group theory and combinatorics \cite{ellis2024product,keller2023sharp,keevash2024largest,keller2024t,green2024new}. Despite its broad importance our understanding of KKL-type theorems on general spaces remains poor, standing as a barrier to further progress on key problems such as the unique games conjecture \cite{barak2018small}.


Toward this end, a series of works over the past several years \cite{kaufman2020high,dikstein2018boolean,bafna2020high,bafna2021hypercontractivity,gur2021hypercontractivity,gaitonde2022eigenstripping,hopkins2025hypercontractivity} developed a new framework for unifying our understanding of Fourier analysis and KKL-type Theorems beyond the cube: \textit{spectral high dimensional expanders} (HDX). HDX are a generalization of expanders to hypergraphs and ranked posets that have seen a recent explosion of application within theoretical computer science \cite{kaufman2020high,dikstein2018boolean,bafna2020high,bafna2021hypercontractivity,gur2021hypercontractivity,gaitonde2022eigenstripping,dinur2017high,dikstein2019agreement,kaufman2020local,alev2019approximating,dinur2020explicit,hopkins2022explicit,Beaglehole2022SamplingEF,anari2019log,alev2020improved,anari2020spectral,chen2020rapid,chen2021optimal,chen2021rapid,feng2021rapid,jain2021spectral,liu2021coupling,blanca2021mixing,dinur2021,panteleev2021asymptotically,dinur2023new,dikstein2024chernoff,dikstein2024low,bafna2024constant,bafna2025quasi,hsieh2025explicit,gur20253,dikstein2026high}. In particular, Bafna, Hopkins, Kaufman, and Lovett \cite{bafna2020high,bafna2021hypercontractivity}, Gur, Lifshitz, and Liu \cite{gur2021hypercontractivity}, and Hopkins \cite{hopkins2025hypercontractivity} proved a KKL-Type Theorem for certain spectral high dimensional expanders stating that any low influence function must be \textit{local}, highly concentrated in some restriction of the hypergraph. Unfortunately, these results only hold for the regime of \textit{near-perfect} expansion, a strong requirement satisfied only by a few families of sparse hypergraphs \cite{lubotzky2005explicit,kaufman2018construction,dikstein2023new}.

In this work, we develop simple local-to-global machinery for analyzing low influence functions on hypergraphs allowing us to remove these requirements in certain well-studied settings. We prove two main results: a KKL-Type Theorem for simplicial complexes whose links (local restrictions) themselves satisfy such a bound such as quotients of the affine building \cite{lubotzky2005explicit} (the so-called Ramanujan complexes), and a weaker (dimension-dependent) KKL-type Theorem for simplicial complexes with \textit{any} non-trivial (two-sided) spectral high dimensional expansion allowing us to give the first non-trivial characterization of low influence functions on product-based HDX \cite{golowich2021improved} and dense clique complexes. We also derive Kruskal-Katona theorems for these objects, a broadly used set of results in extremal combinatorics bounding the lower shadow of a set system.

Our work builds on a growing body of Fourier analytic tools for high dimensional expanders \cite{kaufman2020high,dikstein2018boolean,bafna2020high,bafna2021hypercontractivity,gur2021hypercontractivity,gaitonde2022eigenstripping,dikstein2019agreement,alev2019approximating,gotlib2022fine} and relies in particular on a decomposition theorem of Gotlib and Kaufman (GK) \cite{gotlib2022fine}. We make two main technical contributions in this line. First, we show for any hypergraph $X$ whose \textit{negative} local spectra is bounded away from $-1$, `global' functions on $X$ (i.e.\ functions not concentrated in any local component) have bounded projection onto low levels of the GK-decomposition. Second, we consider an \textit{inductive} approach for bounding expansion through a new `function-dependent' variant of Garland's Lemma, a way to write the expansion of $f$ as an expectation across local parts of the complex where, unlike prior such methods, the underlying distribution may depend on the function $f$ in question. In a similar vein, our local-to-global theorem for (global) small-set expansion combines this latter Garland method with Fourier analytic properties of a localized noise operator to derive global expansion bounds.

\subsection{Background}\label{sec:intro-back}
\paragraph{High Dimensional Expanders:} We focus on weighted \textit{pure simplicial complexes} $(X,\Pi)$ where
\[
X=X(0) \cup \ldots \cup X(d)
\]
for $X(d) \subset {n \choose d}$ an arbitrary $d$-uniform hypergraph and $X(i) \subset {n \choose i}$ given by downward closure, and
\[
\Pi=(\pi_0,\ldots,\pi_d)
\]
for $\pi_d$ an arbitrary distribution over $X(d)$ and $\pi_i$ given by selecting a uniformly random $i$-set from $\tau \sim \pi_d$.

Our bounds will depend on the now standard notion of high dimensional expansion which bounds the behavior of local neighborhoods of $X$ called \textit{links}. For every face $\tau \in X(i)$, the link of $\tau$ is given by:
\[
X_\tau \coloneqq \{\sigma \setminus \tau \in X: \tau \subseteq \sigma \in X\}
\]
Given a function $f: X(k) \to \R$, its localization $f|_\tau:X_\tau(k-|\tau|) \to \R$ is $f_\tau(\sigma)=f(\tau \cup \sigma)$.

Denote the normalized adjacency matrix of the graph underlying $X_\tau$ by $A_\tau$, and for $0 \leq j \leq d-2$ denote the worst-case spectral parameters of $j$-links by
\[
\gamma_j \coloneqq \max_{\tau \in X(j)}\{\lambda_2(A_\tau)\}, \quad \quad \gamma_j^{(-)} \coloneqq \min_{\tau \in X(j)}\{\lambda_{\text{min}}(A_\tau)\}.
\]
A complex is called \textit{strongly connected} if every $\gamma_j < 1$, a (one-sided) \textit{local-spectral expander} if every $\gamma_j$ is bounded away from $1$ \cite{kaufman2016isoperimetric,oppenheim2018local,dinur2017high}, and a (two-sided) local-spectral expander if every $\gamma_j^{(-)}$ is additionally bounded away from $-1$.

\paragraph{Influence and the Down-Up Walk:} As in the seminal work of KKL, we are interested in characterizing the structure of boolean functions $C_k(X,\mathbb{F}_2) \coloneqq \{f: X(k) \to \{0,1\}\}$ with low total influence. In our setting it will be convenient to use an equivalent definition (up to normalization) concerning the expansion of a celebrated random walk on complexes called the \textit{lower} or \textit{down-up} walk. The lower walk at level $k$, denoted $N_k^1$, moves between $k$-faces of the complex via a shared $(k-1)$-face, removing a vertex uniformly at random and re-sampling from the appropriate conditional distribution on the remaining $(k-1)$-face. The \textit{edge expansion} of a set $S \subset X(k)$ with respect to the lower walk is the expected probability of leaving $S$ in a single step, or equivalently:
\[
\Phi(S) = 1-\frac{\langle N_k^11_S, 1_S \rangle}{\langle 1_S,1_S \rangle}.
\]
It is not hard to see that expansion, which can also be written as $\frac{\langle 1_S,L1_S \rangle}{\langle 1_S,1_S \rangle}$ for the `Laplacian Operator' $L$, is equivalent to standard notions of influence up to normalization (see \Cref{sec:background} for details), so we will focus on this notion throughout instead.

\paragraph{Small-Set Expansion and the Noise Operator:} The small-set expansion theorem is a classical and closely related result to the KKL theorem \cite{ahlswede1976spreading,kahn1988influence}. In its standard form, it states that small sets on the noisy cube expand near perfectly. To prove an analog result beyond the cube, we need to define a corresponding notion of the noise operator $T_\rho$ on simplicial complexes. A natural generalization is to consider the walk on $k$-faces of $X$ which fixes each vertex with probability $\rho$, and re-samples the remaining vertices. Formally, we may express this operator as a convex combination of `longer' down-up walks
\[
T_\rho \coloneqq \sum\limits_{i=0}^k {k \choose i}(1-\rho)^{k-i} \rho^{i}N_k^{k-i},
\]
where $N_k^{j}$ moves between $k$-faces via a shared $(k-j)$-face. Note this recovers the standard operator when $X$ is a product space (see \cite{bafna2021hypercontractivity}). We will always write expansion with respect to the noise operator at $\Phi_{T_\rho}$, and write just $\Phi$ for the down-up walk. Finally, it will sometimes be useful to reference the opposite variant of the lower walks, the `upper' walks $\widehat{N}_k^i$, which moves between $k$-faces via a shared $(k+i)$-face.
\subsection{Results}
We are now ready to more formally cover our results. We split the subsection into two parts: structure theorems for low-influence functions on arbitrary HDX, and a local-to-global small-set expansion theorem.
\subsubsection{A Structure Theorem for Low Influence Functions}
The canonical examples of low influence (non-expanding) sets on simplicial complexes are \textit{links}.\footnote{Note here we really mean the set of $k$-faces including some fixed $i$-face, sometimes called a star (rather than removing the common face as in the formal definition of a link).} Our goal is to show these are the \textit{only} such examples. To this end, we call a function $f$ $(\varepsilon,i)$-global if it is uncorrelated with $i$-links:
\[
\forall \tau \in X(i): \underset{X_\tau}{\mathbb{E}}[f_\tau] \leq \mathbb{E}[f]+\varepsilon.
\]
We prove that global functions indeed expand, where the quantitative parameters depend on the underlying local-spectral expansion of the complex.

\begin{theorem}[Expansion of Global Functions (\Cref{thm:expand})]\label{thm:intro-exp}
Let $(X,\Pi)$ be a $k$-dimensional simplicial complex and $f: X(k) \to \mathbb{R}$ any $(\varepsilon,i)$-global boolean function. Then
\[
\Phi(f) \geq \frac{1-\mathbb{E}[f]}{k-i}\prod\limits_{j=i}^{k-2}(1-\gamma_j) - c_{k,i,\gamma}\varepsilon.
\]
where $c_{k,i,\gamma} \leq (k-i+1)\left(\frac{1}{k-i}\prod\limits_{j=i}^{k-2}(1-\gamma_j)-\frac{1}{k-i+1}\prod\limits_{j=i-1}^{k-2}(1-\gamma_j) \right)\left(1+(k-i)\gamma^{(-)}_{i-1}\right)^{-1}$
\end{theorem}
For many regimes of interest, e.g.\ when either $\gamma_{k-2}$ or $|\gamma^{(-)}_{k-2}|$ is less than roughly $\frac{1}{k}$, $c_{\gamma,k,i} \leq O(1)$ is an \textit{absolute constant}, with no dependence on dimension. In such cases the contrapositive of \Cref{thm:intro-exp} states \textit{any non-expanding function has constant density in some link}. We discuss the relation of this bound to prior literature further in \Cref{sec:discussion}. It remains an interesting open problem whether the expression is tight, or could be improved to match the expansion profile known under stronger expansion guarantees.

As applications of \Cref{thm:intro-exp}, we give the first structure theorems for simpler combinatorial constructions of high dimensional expanders. We'll start by looking at one of the most classical constructions of simplicial complexes: \textit{clique-complexes}. Given a graph $G=(V,E)$, the $k$-dimensional clique-complex $K_{G,k}$ is the complex induced by taking the uniform distribution over $k$-cliques of $G$. Using \Cref{thm:intro-exp}, we give the first non-trivial characterization of non-expanding sets on (dense) clique-complexes.
\begin{corollary}[Expansion in Clique-Complexes (\Cref{thm:clique})]\label{cor:intro-clique}
Fix $k \in \mathbb{N}$ and let $G=(V,E)$ be any graph with minimum degree at least $\Delta_{\text{min}} \geq \frac{2k-2}{2k-1}|V|$. Then the expansion of any $(\varepsilon,i)$-global boolean function $f$ on the $k$-dimensional clique-complex $K_{G,k}$ is at least:
\[
\Phi(f) \geq (1-\mathbb{E}[f])\frac{(i+1)}{k(k-i)} - O(\varepsilon)
\]
\end{corollary}
While \Cref{cor:intro-clique} is non-trivial in any dimension, it is strongest in low dimensions. For instance, an interesting implication for $3$-clique (triangle) complexes is that any function of triangles with expansion worse than $1/3$ must have a dense vertex.
\begin{corollary}[Low Influence Functions on Triangle Complexes (\Cref{cor:clique})]
Let $G=(V,E)$ be any graph of minimum degree at least $\Delta_{\text{min}} \geq \frac{5}{6}|V|$ and $S \subset K_{G,3}(3)$ any set with expansion
\[
\Phi(S) \leq \frac{1}{3}-\delta.
\]
Then there exists a vertex $v \in G$ such that $S$ contains at least a $\frac{\delta}{2}$-fraction of the triangles touching $v$.
\end{corollary}

Of course, clique-complexes arising from dense graphs are themselves dense objects. A major selling points of prior work was the ability to characterize low-influence functions on several families of \textit{sparse} complexes. While our work cannot match the quantitative strength of these results, it does lead to new insight on the structure of functions on simpler constructions of bounded-degree HDX with quantitatively weaker parameters. In particular, we consider (cutoffs of) the elementary combinatorial HDX of Golowich \cite{golowich2021improved}, which tensors a bounded-degree expander $G$ with the complete complex to give a weak bounded-degree two-sided HDX.
\begin{corollary}[Expansion in Combinatorial HDX (\Cref{cor:prod})]
Let $X$ be a $2k$-dimensional `product-HDX' of any graph $G$ and $\Delta_{4k}(2k)$, and $f \in C_k$ any $(\varepsilon,i)$-global boolean function. Then
\[
\Phi(f) \geq \frac{(1-\mathbb{E}[f])i}{(k-1)(k-i)} - O(\varepsilon)
\]
\end{corollary}
As for the case of clique-complexes, this is especially powerful in low dimensions. For concreteness, we again look at the first non-trivial setting of $k=3$.
\begin{corollary}[Low Influence Functions on Combinatorial HDX (\Cref{cor:prod2})]
Let $X$ be a $3$-dimensional `product-HDX' of any graph $G$ and $\Delta_n(6)$, and $S \subset X(3)$ be any subset of triangles with expansion at most
\[
\Phi(S) \leq \frac{1}{4} - \delta.
\]
Then there exists a vertex $v \in X(1)$ such that $S$ contains a $\Omega(\delta)$ fraction of triangles including the vertex $v$.
\end{corollary}
Finally, we note that as an immediate corollary of \Cref{thm:intro-exp} we also get a `Kruskal-Katona' type theorem in all the above settings. Kruskal-Katona is a seminal and broadly used result in extremal combinatorics which given a set $S \subset X(k)$, lower bounds the number of $(k-1)$-faces that sit inside $S$, called its `lower shadow' and denoted $\partial S$.
\begin{corollary}[Kruskal-Katona for Weak HDX (\Cref{cor:kruskal})]\label{cor:intro-kruskal}
    Let $(X,\Pi)$ be a $k$-dimensional simplicial complex and $S \subset X(k)$ any $(\varepsilon,i)$-global set. Then
\[
\mathbb{E}[\partial S] \geq \mathbb{E}[S]\left(1+\frac{1-\mathbb{E}[S]}{k-i}\prod\limits_{j=i}^{k-2}(1-\gamma_j) - c_{k,i,\gamma}\varepsilon.\right)
\]
where $c_{k,i,\gamma} \leq (k-i+1)\left(\frac{1}{k-i}\prod\limits_{j=i}^{k-2}(1-\gamma_j)-\frac{1}{k-i+1}\prod\limits_{j=i-1}^{k-2}(1-\gamma_j) \right)\left(1+(k-i)\gamma^{(-)}_{i-1}\right)^{-1}$
\end{corollary}

\subsubsection{Small-Set Expansion}
Similar to the setting of the lower walk, links also provide a canonical family of non-expanding sets for the noise operator on simplicial complexes. In this setting, it is typical \cite{keevash2019hypercontractivity,gur2021hypercontractivity,bafna2021hypercontractivity} to use a slightly stronger notion of global functions we'll call `$(\varepsilon,i)$-strongly global', which simply assumes $f$ is sparse on every link:
\[
\forall \tau \in X(i): \underset{X_\tau}{\mathbb{E}}[f_\tau] \leq \varepsilon.
\]
BHKL \cite{bafna2021hypercontractivity}, GLL \cite{gur2021hypercontractivity}, and Hopkins \cite{hopkins2025hypercontractivity} prove small-set expansion for strongly global functions on two-sided and one-sided partite HDX with very strong quantitative expansion.\footnote{In particular, they require local-spectral expansion to be inverse exponential in the dimension.} While these properties are satisfied by some constructions (e.g.\ \cite{kaufman2020high}), some variants of the seminal \textit{Ramanujan complexes} (quotients of the affine building) of Lubotzky, Samuels, and Vishne \cite{lubotzky2005explicit} actually fail these conditions.

On the other hand, these complexes exhibit nice \textit{local} structure. In particular, for large enough fields, their links are quantitatively strong partite HDX, and therefore satisfy a small-set expansion theorem. This raises a natural question: is any complex $X$ whose links satisfy global small-set expansion itself a global small-set expander? We answer this question in the affirmative.
\begin{theorem}[Local-to-Global SSE (Informal \Cref{thm:local-to-global})]\label{thm:intro-local-to-global}
    Let $X$ be a complex such that $\forall v \in X(1)$:
    \begin{enumerate}
        \item Any $(\varepsilon,i-1)$-strongly global function $f$ on $X_v$ has expansion
        \[
        \Phi_{T_\rho}(f) \geq \phi(\varepsilon,k,i)
        \]
        \item $X_v$ is a $\gamma$-two-sided or $\gamma$-one-sided partite HDX
    \end{enumerate}
    Then any $(\varepsilon,i)$-global boolean function $f$ has expansion at least
    \[
    \Phi_{T_\rho}(f) \geq \phi(\varepsilon,k,i) - 2^{O(k)}\gamma - 2^{-\Omega(k)}.
    \]
\end{theorem}
We note the requirement that links of $X$ are HDX can be removed up to worsened dependence in the dimensional error term (see \Cref{thm:local-to-global}). The inverse exponential error above is essentially negligible, since any function with respect to the noise operator has non-expansion at least inverse exponential in $k$. 

As an immediate corollary, we prove the Ramanujan complexes are global small-set expanders.
\begin{corollary}[Ramanujan Complexes are Global SSEs (Informal \Cref{cor:LSV-expand})]
    Let $f \in C_k$ be an $(\varepsilon,i)$-strongly global boolean function on a Ramanujan complex of \cite{lubotzky2005explicit} with $2^{-\Omega(k)}$ one-sided local-spectral expansion. The expansion of $f$ is at least
    \[
    \Phi_{T_\rho}(f) \geq 1 - \rho^{i} - O_i(\varepsilon) - 2^{-\Omega(k)}.
    \]
\end{corollary}
\subsection{Techniques}
We sketch the proofs of \Cref{thm:intro-exp} and \Cref{thm:intro-local-to-global}, starting with the former.

\paragraph{Proof Overview of \Cref{thm:intro-exp}} Our general approach to prove global sets expand is broken into three core components. The first is a new function-dependent variant of `Garland's Lemma', a classical method of breaking global functions into local components over the complex. In particular, we show the expansion of a function $f$ with respect to the lower walks can be written as an expectation of its expansion over links of $X$.
\begin{lemma}[Garland's Lemma for Expansion (Informal \Cref{lemma:local-to-global-expansion}]\label{lem:intro-Garland}
    Let $(X,\Pi)$ be a weighted, pure simplicial complex and $f \in C_k$ any function on $k$-faces. Then for all $i \leq k$ and $j \leq k-i$:
\begin{equation}
\Phi(f) = \underset{\tau \sim \pi^f_{j}}{\mathbb{E}}[\Phi_{X_\tau}(f|_\tau)],
\end{equation}
where $f|_\tau(\sigma)=f(\tau \cup \sigma)$ is the localization of $f$ to the link of $\tau$, and $\pi^f_{j}$ is some $f$-dependent distribution over $j$-faces of $X$.
\end{lemma}
We remark that to the authors' knowledge, this is the first variant of Garland's lemma where the expectation is over a \textit{function-dependent} distribution, rather than simply over the complex distribution $\pi_j$, which is necessary for our setting.

The second core component is a variant of the GK's decomposition theorem \cite{gotlib2022fine} for the lower walk.
\begin{theorem}[GK-Decomposition (Informal \Cref{thm:decomp-app})]
Let $(X,\Pi)$ be a $k$-dimensional simplicial complex. Then for any $f \in C_{k}$ there is an orthogonal decomposition $f=\sum\limits_{i=0}^k f_i$ such that:
\[
\langle N^k_1 f,f \rangle \leq \mathbb{E}[f]^2 + \sum\limits_{i=1}^{k-1}\left(1-\frac{1}{k-i+1}\prod\limits_{j=i-1}^{k-2}(1-\gamma_j) \right) \langle f,f_i \rangle.
\]
\end{theorem}
Finally, the third main component is a `level-$1$ inequality' showing any $(\varepsilon,1)$-global function has low projection onto the first level of the GK-decomposition.
\begin{proposition}[Level-$1$ Inequality (Informal \Cref{cor:level-1})]\label{intro:level-i}
Let $(X,\Pi)$ be a $k$-dimensional complex and $f$ any $(\varepsilon,1)$-global boolean function $f \in C_k$. Then
\[
\langle f, f_1 \rangle \leq \frac{k}{1+\gamma_0^{(-)}(k-1)}\varepsilon.
\]
\end{proposition}
We remark that this is a corollary of a more general level $i$ inequality that scales with the minimum non-zero eigenvalue of the upper walks (see \Cref{thm:i-inequality}). One could instead apply this directly with the GK decomposition, but the resulting dependence on $\varepsilon$ is worse than can be achieved by combining these with Garland's method.

Given these three components, \Cref{thm:intro-exp} follows from a careful induction. We first prove the statement directly for $(\varepsilon,1)$-global functions using the GK-Decomposition and Level-1 inequality. For general $i$, we use Garland's lemma to reduce to $i=1$ inside links of $X$, writing the expansion as $\Phi(f) = \mathbb{E}_{\tau \sim \pi_{i-1}^f}[\Phi_\tau(f|_\tau)]$, and observe that the localization $f|_\tau$ of a $(\varepsilon,i)$ global function is itself $(\varepsilon-\delta_\tau,1)$ global for some $|\delta_\tau| \leq \varepsilon$. Applying the statement for $1$-global functions inside the expectation (carefully handling dependence on $\mathbb{E}_{\tau \sim \pi_{i-1}^f}[\delta_\tau]$) then gives the claimed bound.

\paragraph{Proof Overview of \Cref{thm:intro-local-to-global}:} The first main component of our low influence characterization, \Cref{lem:intro-Garland}, already shows that expansion of the lower walk exhibits local-to-global behavior. The proof of \Cref{thm:intro-local-to-global} would follow immediately if the same form of localization were to hold for the noise operator. Unfortunately, while the lemma may be extended to lower walks of arbitrary length, such a statement does not hold directly for the noise operator.

To see why, recall the noise operator is a convex combination of lower walks $T_\rho = \sum\limits_{i=0}^k B_k^\rho(i)N_k^{k-i}$ where $B_k^\rho(i)={k \choose i}(1-\rho)^{k-i}\rho^{i}$. Since expansion is linear, for any function $f$ we can write
\[
\Phi_{T_\rho}(f) = \sum\limits_{i=0}^k B_k^\rho(i)\Phi_{N_k^{k-i}}(f).
\]
Localizing $T_\rho$ naively hits two main issues. First, a walk $N_k^{i}$ can only be appropriately localized to $j$-links when $j \leq k-i$ (otherwise the walk goes below the level of the link, and cannot be captured at this level of locality). Second, the \textit{coefficients} of $T_\rho$, which are binomially distributed, depend on dimension and therefore do not localize correctly. In particular, after localizing `permissible' lower walks in the sum, we'd get:
\[
\Phi_{T_\rho}(f) = \sum\limits_{i=0}^{j-1} B_k^\rho(i)\Phi_{N_k^{i}}(f) + \mathbb{E}_{\tau \sim \pi^f_j}\left[\sum\limits_{i=j}^k B_k^\rho(i)\Phi_{N_{k-j}^{k-i}}(f|_\tau).\right]
\]
To correspond to the local noise operator on the right-hand side, we'd instead need the corresponding coefficients to be $B_{k-j}^\rho(i)$, the binomial distribution on $k-j$ trials, not $k$ trials as in the original instance. A first approach to solving these issues is to define a `shifted' noise operator over the original space that simply throws out the lefthand sum, which is negligible, and replaces the remaining coefficients with $B^\rho_{k-j}(i)$ over the original space. One can then argue that the expansion with respect to the shifted operator `localizes correctly' (allowing us to apply local SSE), and is moreover close to the expansion of the original noise operator. Unfortunately, this procedure results in error scaling with the TV-distance between Binomial distributions on $k$ and $k-j$ trials. This costs $\poly(k^{-1})$ additive error, and therefore does not give the desired global SSE theorem.

To fix this issue, we instead make the stronger assumption that links of $X$ are good two-sided or partite one-sided high dimensional expanders, a property satisfied by our main motivating example (the Ramanujan complexes). This allows us to appeal to the theory of Fourier analysis developed in \cite{dikstein2018boolean,bafna2020high,bafna2021hypercontractivity,gur2021hypercontractivity}, which shows one may decompose expansion with respect to the lower walks as
\[
\Phi_{N_k^i}(f) \approx \frac{1}{\langle f,f \rangle}\sum\limits_{\ell=0}^k \lambda(N^i_k)\langle f_\ell,f_\ell \rangle
\]
for some decomposition $f = \sum\limits_{i=0}^k f_i$ and approximate eigenvalues $\lambda(N^i_k)$. With this in mind, instead of shifting the noise operator, we can decompose the localized expansion in the Fourier basis as
\[
\Phi_{T_\rho}(f) \lesssim \sum\limits_{i=0}^{j-1} B_k^\rho(i) +  \underset{\tau \sim \pi_j^f}{\mathbb{E}}\left[\frac{1}{\langle f|_\tau,f|_\tau \rangle}\sum\limits_{i=j}^k B_k^\rho(i)\sum\limits_{\ell=0}^{k-j}\lambda_\ell(N_{k-j}^{k-i}|_\tau)\langle (f|_\tau)_\ell,(f|_\tau)_\ell \rangle\right]
\]
Finally for each fixed $\ell$, we show its corresponding coefficient in the above sum is at most $\rho^\ell$, the (approximate) eigenvalue of the local noise operator itself \cite{bafna2021hypercontractivity,gur2021hypercontractivity}. This allows us to write the desired localized inequality
\[
\Phi_{T_\rho}(f) \lesssim \sum\limits_{i=0}^{j-1} B_k^\rho(i) +  \underset{\tau \sim \pi_j^f}{\mathbb{E}}\left[\Phi_{T_\rho}(f|_\tau)\right]
\]
which may finally be bounded by the local SSE theorem using the fact that any $j$-restriction $f|_\tau$ of an $(\varepsilon,i)$-strongly global function is $(2\varepsilon,i-j)$-strongly global. The error in this analysis scales only with the cost of cutting off small $i$ (the first term) and the approximation error of the Fourier decomposition. For constant $j$, the former is at most $\exp(-k)$, while the latter is $2^{O(k)}\gamma$ when $j$-links are $\gamma$-HDX. Thus for complexes with sufficiently expanding links, we get an SSE theorem with negligible additive error as desired.

\subsection{Discussion and Related Work}\label{sec:discussion}

We close the section with a more in-depth comparison with prior work and a natural open problem such comparisons raise. At a high level, our work fits into a long line of research examining the structure of low-influence functions beyond the hypercube, starting not long after the KKL Theorem itself with the study of product spaces and the $p$-biased cube \cite{bourgain1992influence,talagrand1994russo,friedgut1996every,friedgut1998boolean,hatami2012structure,keevash2019hypercontractivity} and more recently in extended settings such as the slice \cite{khot2018small}, multi-slice \cite{filmus2018log,salez2020sharp}, symmetric group \cite{filmus2020hypercontractivity}, Grassmannian \cite{subhash2018pseudorandom,ellis2022forbidden}, high dimensional expanders \cite{bafna2020high,bafna2021hypercontractivity,gur2021hypercontractivity,gaitonde2022eigenstripping, gotlib2022fine}, and Lie groups \cite{ellis2024product}. We rely heavily on the Fourier analytic machinery initiated by Dikstein, Dinur, Filmus, and Harsha \cite{dikstein2018boolean}, and Kaufman and Oppenheim \cite{kaufman2020high}, and further extended by Bafna, Hopkins, Kaufman, and Lovett \cite{bafna2020high,bafna2021hypercontractivity}, Gur Lifshitz, and Liu \cite{gur2021hypercontractivity}, Hopkins \cite{hopkins2025hypercontractivity}, and Gotlib and Kaufman \cite{gotlib2022fine}.

Concretely, several bounds are known for the expansion of global sets with respect to the lower walk. In the regime of strong (inverse exponential) quantitative expansion, BHKL \cite{bafna2020high} prove any $(\varepsilon,i)$-global function has expansion at least
\[
\Phi(f) \gtrsim (1-\mathbb{E}[f])\frac{i+1}{k} - {k \choose i}\varepsilon.
\]
This is tight when $\varepsilon \ll \frac{1}{{k \choose i}}$, and mimics the structure one would expect to see on general $k$-fold product spaces. For strongly global functions, BHKL \cite{bafna2021hypercontractivity}, GLL \cite{gur2021hypercontractivity}, and Hopkins \cite{hopkins2025hypercontractivity} improve the $\varepsilon$-dependence to be \textit{independent of dimension}:
\[
\Phi(f) \gtrsim (1-\mathbb{E}[f])\frac{i+1}{k} - 2^{O(i)}\varepsilon.
\]
In the regime of weak high dimensional expansion, the only known bound for global functions is due to Gotlib and Kaufman \cite{gotlib2022beyond}, who show that on any \textit{one-sided} HDX:
\[
    \Phi(f) \geq (1-\mathbb{E}[f])\frac{1}{k-i}\prod\limits_{j=i}^{k-2}(1-\gamma_j) + |X(k)|\varepsilon.
\]
Unlike the prior bounds, however, this is only really meaningful when $\varepsilon \approx 0$, a non-trivial but extremely strong notion of globalness that does not lead to meaningful structural characterizations.

In some sense our bound, \Cref{thm:intro-exp}, sits between the latter two. it has the beneficial properties of the latter in that it holds for global functions and under arbitrary (two-sided) local-spectral expansion, and of the former in that the dependence on $\varepsilon$ is typically constant. On the other hand, due to being obtained through an inductive approach, it also inherits the main term of Gotlib-Kaufman \cite{gotlib2022fine}, scaling with $\frac{1}{k-i}$ instead of $\frac{i+1}{k}$. When $i$ and $k$ are large, the former dependence is substantially worse.

At a technical level, this difference in main term stems from the fact that our `advantage' for $i$-global functions, like in Gotlib-Kaufman \cite{gotlib2022fine}, comes from the second eigenvalue of the lower walk inside $i$-links of $X$, which scales as $1-\frac{1}{k-i+1}\prod\limits_{j=i-1}^{k-2}(1-\gamma_j)$. On the other hand, on very strong HDX, the advantage of $i$-global functions comes from the $i$th approximate eigenvalue of the global lower walk on $X$ itself, which scales with $1-\frac{i}{k}$. It is natural to conjecture the `correct' dependence sits between these two bounds, as $1-\frac{i}{k}\prod\limits_{j=i-1}^{k-2}(1-\gamma_j)$, closer to the type of bound known for the second eigenvalue from \cite{alev2020improved}. Unfortunately, both the techniques of this paper and Fourier analytic techniques from prior work (which typically incur exponential error) seem far from proving such a bound.

\section{Background}\label{sec:background}
We now give a more formal introduction to the language of high dimensional expansion, including simplicial complexes, local-spectral expansion, high order random walks, and influence/expansion of sets.
\subsection{Simplicial Complexes and Local-Spectral Expansion}
We study the structure of low influence functions on weighted \textit{pure simplicial complexes}.
\begin{definition}[Simplicial Complex]
A weighted, $d$-dimensional\footnote{We note that our notion of dimension is off by $1$ from much of the historical HDX literature, but generally leads to simpler expressions in our setting.} pure simplicial complex $(X,\Pi)$ on $n$ vertices consists of a complex
\[
X=X(0) \cup \ldots \cup X(d)
\]
where $X(d) \subseteq {[n] \choose d}$ is a $d$-uniform hypergraph and each $X(i) \subseteq {[n] \choose i}$ is given by downward closure:
\[
X(i) \coloneqq \left\{ \sigma \in {[n] \choose i}: \exists \tau \in X(d), \sigma \subset \tau \right\},
\]
and a joint distribution $\Pi=(\pi_1,\ldots,\pi_d)$ over $X(0) \times \ldots \times  X(d)$ where $\pi_i$ is induced from $\pi_d$ by sampling a uniformly random size-$i$ subset:
\[
\pi_i(\sigma) = \frac{1}{{d \choose i}}\sum\limits_{\tau \supset \sigma: |\tau|=d} \pi_d(\tau) .
\]
\end{definition}
Since all complexes considered in this work will be pure, weighted, and $d$-dimensional, we drop these monikers throughout and usually refer to $(X,\Pi)$ as just a simplicial complex. We write $\Delta_n$ to denote the complete simplicial complex on $n$ vertices.

We will sometimes work in the special setting of \textit{partite} complexes, which often play a special role in the theory of high dimensional expansion.
\begin{definition}[Partite Complex]
    A simplicial complex $(X,\Pi)$ is called partite if its vertices can be partitioned into $d$ `colors':
    \[
    X(1) = X^1 \amalg \ldots \amalg X^d
    \]
    such that every $d$-face has exactly one vertex from each color.
\end{definition}

There are many notions of high dimensional expansion on simplicial complexes. In this work we take the recently popular \textit{local-spectral} approach that analyzes the spectrum of local components called \textit{links}.
\begin{definition}[Link]
Given a complex $(X,\Pi)$ and a face $\tau \in X(i)$, the link of $\tau$ is the $(d-i)$-dimensional sub-complex $(X_\tau,\Pi_\tau)$ where
\[
X_\tau \coloneqq \{\sigma \setminus \tau \in X: \tau \subseteq \sigma \in X\}
\]
and $\Pi_\tau = (\pi_{\tau,0},\ldots, \pi_{\tau,d-i})$ is the distribution induced by restricting $\Pi$ to $X_\tau$ in the natural manner
\[
\pi_{\tau,d-i}(\sigma) = \frac{\pi_d(\sigma \cup \tau)}{\sum\limits_{\sigma' \in X_\tau}\pi_{d}(\sigma' \cup \tau)}.
\]
\end{definition}
Typically, a complex is said to be a \textit{local-spectral expander} if the graph underlying each link of co-dimension at least two is a spectral expander \cite{oppenheim2018local,kaufman2020high}. 
\begin{definition}[Underlying Graph]
Given a simplicial complex $(X,\Pi)$ and a link $(X_\tau,\Pi_\tau)$, the (edge)-weighted graph underlying $X_\tau$ is given by $G_{X_\tau}=(X_\tau(1),X_\tau(2),\pi_2)$, and has weighted adjacency (random walk) matrix
\[
A_\tau(v, w) = \pi_{\tau \cup v,1}(w).
\]
\end{definition}
In our work, it will be convenient to have finer grain control over the local spectral structure of these operators than is typically considered in the literature.
\begin{definition}[Local Spectra]
Let $(X,\Pi)$ be a simplicial complex. We define the local-spectral parameters $\{\gamma_i\}_{i=0}^{d-2}$ and $\{\gamma_i^{(-)}\}_{i=0}^{d-2}$ to be the worst non-trivial positive and negative eigenvalues respectively across links of each dimension: 
\[
\gamma_i \coloneqq \max_{\tau \in X(j)}\{0,\lambda_2(A_\tau)\}, \quad \quad \gamma_i^{(-)} \coloneqq \min_{\tau \in X(j)}\{\lambda_{\text{min}}(A_\tau)\}.
\]
\end{definition}
In more standard language, an infinite family of complexes are called \textit{one-sided} local-spectral expanders \cite{oppenheim2018local} if every $\gamma_i$ is bounded away from $1$, and \textit{two-sided} local-spectral expanders \cite{dinur2017high} if additionally each $\gamma_{i}^{(-)}$ is also bounded away from $-1$.
\subsection{Averaging Operators and the KO-Decomposition}
Given a complex $(X,\Pi)$, let $C_k = C_k(X,\mathbb{R}) \coloneqq \{f: X(k) \to \mathbb{R}\}$ denote the space of real-valued functions over $k$-sets of $X$. The function spaces associated with any complex come with a set of standard analysis tools, including a natural weighted inner product:
\[
\langle f,g \rangle \coloneqq \underset{\pi_i}{\mathbb{E}}[fg],
\]
and \textit{averaging operators} (sometimes called `unsigned boundary operators') that map between the $C_i$. In particular, for every $0 \leq i < d$, the \textit{Up Operator} $U_i: C_i \to C_{i+1}$ lifts functions between levels:
\[
\forall \sigma \in X(i+1):~ U_if(\sigma) = \frac{1}{i+1}\sum\limits_{\tau \subset \sigma: ~|\tau| = i} f(\tau),
\]
and for every $0 < i \leq d$ the \textit{Down Operator} $D_i: C_i \to C_{i-1}$ conversely lowers functions:
\[
\forall \sigma \in X(i-1):~ D_{i}f(\sigma) = \sum\limits_{\tau \supset \sigma: ~|\tau|=i}\pi_{\sigma,1}(\tau \setminus \sigma)f(\tau).
\]
It will often be convenient to lift or lower a function multiple levels, which can be done through the following compositions
\[
U^k_i \coloneqq U_{k-1}\circ \ldots \circ U_i, \quad \quad D^k_i \coloneqq D_{i+1}\circ \ldots \circ D_k.
\]
A crucial and well-known fact is that the averaging operators $D$ and $U$ are adjoint.
\begin{lemma}
Let $(X,\Pi)$ be a simplicial complex and $0 \leq k < d$. Then for all $f \in C_k$ and $g \in C_{k+1}$ we have
\[
\langle U_if, g \rangle = \langle f, D_{i+1}g \rangle.
\]
\end{lemma}
The averaging operators also give a natural formalization of what it means for a function to `come from below,' namely being in $\text{Im}(U^k_i)$. This can be used in a number of ways to build natural function decompositions on simplicial complexes that break up a function into contributions from each level. In this work, we rely on the basis developed by Kaufman and Oppenheim \cite{kaufman2020high} which can be stated in terms of the averaging operators as follows.
\begin{definition}[KO-Decomposition \cite{kaufman2020high}]
Let $(X,\Pi)$ be a simplicial complex, $0 \leq k \leq d$, and $f \in C_k$. There exists an orthogonal decomposition $f = \sum\limits_{i=0}^k f_i$ such that each $f_i \in Im(U^k_i) \cap Ker(D^k_{i-1})$.
\end{definition}
\subsection{High Order Random Walks}
Just as expander graphs are inextricably tied to their underlying random walk, high dimensional expanders also bear close connection to high order analogs on their underlying complex. We will focus in particular in this work on two classical settings, the \textit{down-up walks}, and the \textit{noise operator}.
\begin{definition}[Down-up Walk {\cite{kaufman2016high,dinur2017high}}]
Given a complex $(X,\Pi)$, for any $0 < k \leq d$ the down-up walk on $k$-faces, denoted $N_k^1$, walks between $k$-faces of $X$ via a shared $(k-1)$-face, and can be written formally as
\[
N_k^1 \coloneqq U_{k-1}D_k.
\]
\end{definition}
It will also be useful for the sake of analysis to have access to longer versions of the down-up walk, as well as their corresponding up-down walks. With this in mind, define
\[
N_k^i \coloneqq U^k_{k-i}D^k_{k-i}, \quad \quad \widehat{N_k^i} \coloneqq  D^{k+i}_{k}U^{k+i}_{k}
\]
to be the walks that move between $k$-faces via a shared $(k-i)$-face and $(k+i)$-face respectively. Finally, we'll also need the following `non-lazy' variant of the up-down walk
\[
M_k^+ \coloneqq \frac{k}{k+1}D_{k+1}U_k - \frac{1}{k+1}I.
\]
Note that by adjointness of $D$ and $U$, all these walks are self-adjoint and therefore have spectral decompositions.

The lower walk can be phrased as the random process that removes a uniformly random element  $v \in \sigma$, and re-samples a new element conditioned on the remaining face $\sigma \setminus \{v\}$. We will also study a variant of this process called the \textit{noise operator}, which removes \textit{each} element in $\sigma$ with some fixed probability, then jointly re-samples these elements conditioned on the rest.
\begin{definition}[The Noise Operator]
Given a complex $(X,\Pi)$, for any $0 < k \leq d$ and $0 \leq \rho \leq 1$, the $\rho$-noise operator on $k$-faces, denoted $T^k_\rho$, re-samples each vertex with probability $1-\rho$:
    \[
T^k_\rho \coloneqq \sum\limits_{i=0}^k {k \choose i}(1-\rho)^i \rho^{k-i}N_k^i,
\]
\end{definition}
We remark that all the walks above are self-adjoint with respect to the inner product defined in the previous subsection, and therefore have a spectral decomposition.
\subsection{Boolean Function Analysis and Expansion}
The down-up walk and its variants capture a broad variety of structures studied throughout theoretical computer science. On the complete complex it gives the Johnson graphs, on spin-systems it gives the celebrated Glauber Dynamics \cite{anari2020spectral}, and on the natural simplicial representation\footnote{In particular, the complete $d$-partite complex on size-$2$ parts.} of $\{0,1\}^d$ it simply results in the standard hypercube graph \cite{bafna2021hypercontractivity} (up to laziness). It is this final connection which allows us to generalize the classical notion of \textit{total influence} of a boolean function, traditionally defined the hypercube as\footnote{Here $x^{\oplus i}$ denotes the string $x$ with the $i$th bit flipped.}
\[
I(f) = \sum\limits_{i \in d}\mathbb{E}_{x \sim \{0,1\}^d}\left[ \left(\frac{f(x) - f(x^{\oplus i})}{2}\right)^2\right] = d\langle f,Lf \rangle
\]
where $L=I-U_{k-1}D_k$ is the standard Laplacian operator. The latter equality leads to a natural interpretation on general simplicial complexes.
\begin{definition}[Total Influence on Complexes \cite{dikstein2018boolean,bafna2021hypercontractivity,gur2021hypercontractivity}]
Given a complex $(X,\Pi)$ and $0 < k \leq d$, the total influence of a boolean function $f \in C_k$ on $k$-sets is
\[
I(f) \coloneqq k\langle f, (I-U_{k-1}D_k)f \rangle
\]
\end{definition}
The total influence of a boolean function is also classically related to its \textit{combinatorial expansion}.
\begin{definition}[Combinatorial (Edge)-Expansion]
Let $M$ be random walk on universe $\Omega$ with stationary distribution $\pi$. The edge expansion of a subset $S \subset \Omega$ is the expected probability of leaving $S$ in a single step of $M$: 
\[
\Phi_M(1_S) \coloneqq \mathbb{E}_{\pi|_S}\left[M(v,X\setminus S)\right]
\]
where $\pi|_S(x)$ is the natural distribution induced from restricting $\pi$ to $S$ and $M(v,X\setminus S)$ is the total outgoing weight from $v$. It will also sometimes be useful to refer to \textbf{non-expansion}, denoted $\bar{\Phi}_M \coloneqq 1-\Phi_M$.
\end{definition} 
It is a simple exercise to show that the expansion of $S$ can also be written as the Rayleigh quotient of $1_S$ with respect to $I-M$:
\[
\Phi_{M}(1_S) = \frac{\langle 1_S, (I-M)1_S \rangle_\pi}{\langle 1_S,1_S \rangle_\pi}
\]
where $\langle f,g \rangle_{\pi}= \mathbb{E}_\pi[fg]$. With this in mind, it is an easy observation that since the stationary distribution of $N_k^1$ is $\pi_k$, the total influence of a function is (up to normalization) exactly its expansion with respect to the down-up walk. We will also frequently use the equiavlent forms

\[
\Phi_{M}(1_S) = 1- \frac{\langle 1_S, M1_S \rangle_\pi}{\langle 1_S,1_S \rangle_\pi} = 1- \frac{1}{\mathbb{E}[1_S]}\langle 1_S, M1_S \rangle_\pi,
\]
and write the \textit{non}-expansion as
\[
\bar{\Phi}_{M}(1_S) = \frac{\langle 1_S, M1_S \rangle_\pi}{\langle 1_S,1_S \rangle_\pi} = \frac{1}{\mathbb{E}[1_S]}\langle 1_S, M1_S \rangle_\pi
\]
Expansion of the noise operator is also a classical object of study in boolean function analysis, where it is sometimes referred to as a function's \textit{noise-sensitivity}. It is a classical result that sparse functions on the hypercube are noise-sensitive
\begin{theorem}[]
    Let $X=\{0,1\}^k$ be the hypercube complex. Then for any $0 \leq  \rho \leq 1$ and $S \subset X(k)$:
    \[
    \Phi_{T_\rho}(S) \geq 1 - \mathbb{E}[1_S]^{\frac{1-\rho}{1+\rho}}
    \]
\end{theorem}
This result is also called the \textit{small-set expansion theorem}, as it states that small sets on the noisy hypercube expand near perfectly.
\subsection{Global Functions}
The small-set expansion theorem fails inherently over unbalanced domains, even in settings as simple as the $p$-biased cube. In particular, `local' functions like indicators or juntas which would be large sets on the cube become small, but maintain their otherwise poor expansion. A long and influential line of work has studied the extension of small-set expansion theorems to this and other settings beyond the cube by showing these are the \textit{only} bad examples. To this end, we call a function \textit{global} if it does not become much denser upon any restriction.
\begin{definition}[Global Function]
    A function $f: X(k) \to \{0,1\}$ is called $(\varepsilon,i)$ global if
\[
\forall \tau \in X(i): \underset{X_\tau}{\mathbb{E}}[f_\tau] \leq \mathbb{E}[f]+\varepsilon.
\]
\end{definition}
KKL-type theorems for global functions are sometimes called `booster' theorems, since they state any function with low total influence must see a density `bump' or `boost' inside some small restriction. This is the content, for instance, of Bourgain's influential Booster Theorem used in analysis of sharp thresholds \cite{friedgut1999sharp}.

For the case of small-set expansion, a more typical definition in the literature beyond the cube is to require the stronger guarantee that the set is not strongly correlated with any restriction. We call such functions \textit{strongly} global.
\begin{definition}[Strongly Global Function]
    A function $f: X(k) \to \{0,1\}$ is called $(\varepsilon,i)$-strongly global if
\[
\forall \tau \in X(i): \underset{X_\tau}{\mathbb{E}}[f_\tau] \leq \varepsilon.
\]
\end{definition}
Small-set expansion for strongly global functions roughly translates to statements of the form:
\[
\Phi_{T_\rho}(S) \geq 1 - \rho^i -c_i\varepsilon^{O_\rho(1)}.
\]
In other words, any function with expansion substantially worse than $1-\rho^i$ must be dense inside some $i$-link. 


\subsection{Localization, Garland's Method, and the Trickling-Down Theorem}
One of the most powerful tools in the study of high dimensional expansion is the idea of localization: breaking analysis of global functions or properties into localized parts. Given $f \in C_k$ and $\tau \in X(i)$, we define the \textit{localization} of $f$ to $\tau$, $f|_\tau: X_\tau(k-i) \to \R$, as
\[
\forall \sigma \in X_\tau(k-|\tau|): f|_\tau(\sigma) = f(\sigma \cup \tau).
\]
We will similarly need to consider localizations of the averaging operators. For any $\tau \in X(i)$, we write $U|_\tau$ and $D|_\tau$ to denote the up and down operators respectively on the link $(X_\tau, \Pi_\tau)$, and similarly denote the localized upper and lower walks as $N^i_{k-|\tau|}|_\tau$ and $\widehat{N}^i_{k-|\tau|}|_\tau$.

The key observation, often referred to as Garland's method,\footnote{In reference to Garland's original work \cite{garland1973p} using similar ideas far before the invent of high dimensional expansion.} is that the natural inner product on $X$ `respects localization.' We will use the following three formalizations of this notion which can be found across a variety of works \cite{oppenheim2018local,kaufman2020high,dikstein2018boolean}.\footnote{We note that some of these equalities hold more generally, but the above is sufficient for our purposes.}
\begin{lemma}[Garland's Method \cite{oppenheim2018local,kaufman2020high,dikstein2018boolean}]\label{lemma:garland}
On any simplicial complex $(X,\Pi)$:
\begin{align}
    \forall 0 \leq i \leq k \leq d~\text{and}~f,g \in C_k&: \langle f, g \rangle = \underset{\tau \sim \pi_i}{\mathbb{E}}\left[\langle f|_\tau, g|_\tau \rangle\right]\label{eq:gar-1}\\
    \forall 0 < k \leq d~\text{and}~ f\in C_k&: \langle N_k^1 f,f \rangle = \underset{v \sim \pi_1}{\mathbb{E}}\left[\langle N_{k-1}^1|_v f|_v, f|_v \rangle\right]\label{eq:gar-2}\\
    \forall 0 \leq k < d~\text{and}~ f\in C_k&: \langle M_k^+ f, f \rangle= \underset{\tau \sim \pi_{k-1}}{\mathbb{E}}\left[\langle A_\tau f|_\tau, f|_\tau \rangle\right] \label{eq:gar-3}
\end{align}
\end{lemma}
The first modern application of Garland's method (albeit applied to a different type of localization) is Oppenheim's Trickling-Down Theorem \cite{oppenheim2018local}, which shows the spectral behavior of co-dimension $2$ links is inherited throughout the rest of the complex.
\begin{theorem}[Trickling-Down \cite{oppenheim2018local}]\label{thm:trickling}
Let $(X,\Pi)$ be a strongly connected complex. Then for all $0 \leq i \leq d-2$:
\begin{enumerate}
    \item $\gamma_{i} \leq \frac{\gamma_{d-2}}{1-(d-2-i)\gamma_{d-2}}$
    \item $\gamma^{(-)}_i \geq \frac{\gamma^{(-)}_{d-2}}{1-(d-2-i)\gamma^{(-)}_{d-2}}.$
\end{enumerate}
Moreover the latter holds without the assumption of strong connectivity.
\end{theorem}
We will make liberal use of Garland's method and the Trickling-Down Theorem throughout.

\section{A Booster Theorem for Arbitrary HDX}
Our first application of the basic inductive method is a weak KKL-type theorem for \textit{arbitrary} (two-sided) HDX. In particular, we will show that any function with low total influence on a complex whose local-spectral expansion is bounded away from $1$ must be \textit{local} (in particular, must have a booster). We first state the result in the contrapositive, that global sets expand.
\begin{theorem}[Global Sets Expand]\label{thm:expand}
Let $(X,\Pi)$ be a simplicial complex, $0 \leq k \leq d$, and $f \in C_k$ any $(\varepsilon,i)$-global boolean function. Then the expansion of $f$ with respect to the lower walk is at least
\[
\Phi(f) \geq \frac{1-\mathbb{E}[f]}{k-i}\prod\limits_{j=i}^{k-2}(1-\gamma_j) - c_{k,i,\gamma}\varepsilon
\]
where $c_{k,i,\gamma} \leq (k-i+1)\left(\frac{1}{k-i}\prod\limits_{j=i}^{k-2}(1-\gamma_j)-\frac{1}{k-i+1}\prod\limits_{j=i-1}^{k-2}(1-\gamma_j) \right)\left(1+(k-i)\gamma^{(-)}_{i-1}\right)^{-1}$
\end{theorem}
While the constant $c_{k,i,\gamma}$ in \Cref{thm:expand} is somewhat hard to interpret a priori, we emphasize that for many reasonable settings $c_{k,i,\gamma} \leq O(1)$ is an \textit{absolute constant}, with no dependence on the dimension. This is the case, for instance, if $|\gamma_{i-1}^{(-)}| \ll \frac{1}{k-i}$,\footnote{Note $|\gamma_{i-1}^{(-)}|$ is always at most $\frac{1}{k-i}$ by Trickling-Down, so this is not such a strong assumption.} or when $\gamma_{k-2} \leq \frac{1}{k-1}$ and $\gamma_{d-2}^{(-)}$ is bounded away from one.


\Cref{thm:expand} is strongest in low dimensions due to the leading coefficient, but still has the following interesting interpretation in the high dimensional regime: any low-influence function must have a \textit{constant} booster inside some link of reasonably large co-dimension.
\begin{corollary}[Booster Theorem for general HDX]
There is a universal constant $c>0$ such that for any $i \in \mathbb{N}$, $\frac{1}{i} > \delta > 0$, and $k$-dimensional complex $(X,\Pi)$, any subset $S \subset X(k)$ with expansion at most:
\[
\Phi(S) \leq \frac{1-\mathbb{E}[1_S]}{i}\prod\limits_{j=k-i}^{k-2}(1-\gamma_j) - \delta 
\]
has constant correlation with a link of co-dimension $i$:
\[
\exists \tau \in X(k-i): \underset{X_\tau}{\mathbb{E}}[1_S] \geq \mathbb{E}[1_S]+ c'\delta
\]
where $c' \leq O(c_{k,k-i,\gamma}^{-1})$.
\end{corollary}
We remark this is actually not implied by prior expansion theorems for the lower walk, even on stronger HDX, which would instead roughly imply the existence of an $O_\delta(k)$-link with $2^{-O_\delta(k)}$ density.

Finally, GLL observe that a bound on expansion of the lower walk of any simplicial complex $X$ immediately implies a Kruskal-Katona Theorem (see \cite[Section 8.3]{gur2021hypercontractivity}). Recall that give a set $S \subset X(k)$, the lower shadow of $S$, denoted $\partial S$, is the set of all $(k-1)$-faces $\tau$ contained in some $s \in S$. We get the following corollary for general simplicial complexes.
\begin{corollary}[{\Cref{cor:intro-kruskal}} Restated]\label{cor:kruskal}
    Let $(X,\Pi)$ be a $k$-dimensional simplicial complex and $S \subset X(k)$ any $(\varepsilon,i)$-global set. Then
\[
\mathbb{E}[\partial S] \geq \mathbb{E}[S]\left(1+\frac{1-\mathbb{E}[S]}{k-i}\prod\limits_{j=i}^{k-2}(1-\gamma_j) - c_{k,i,\gamma}\varepsilon.\right)
\]
where $c_{k,i,\gamma} \leq (k-i+1)\left(\frac{1}{k-i}\prod\limits_{j=i}^{k-2}(1-\gamma_j)-\frac{1}{k-i+1}\prod\limits_{j=i-1}^{k-2}(1-\gamma_j) \right)\left(1+(k-i)\gamma^{(-)}_{i-1}\right)^{-1}$
\end{corollary}


\subsection{Garland's Lemma for Expansion}
The first key component of \Cref{thm:expand} is a new variant of Garland's Lemma. While the tool is already broadly used in the study of expansion of walks on HDX (see e.g.\ \cite{dikstein2018boolean,bafna2020high,bafna2021hypercontractivity}), it is typically applied \textit{indirectly} throughout the analysis. We rely on the key (albeit basic) insight that Garland's Method also applies directly to expansion itself.
\begin{lemma}[Garland's Lemma for Expansion]\label{lemma:local-to-global-expansion}
Let $(X,\Pi)$ be a weighted, pure simplicial complex and $f \in C_k$ any function on $k$-faces. Then the (non)-expansion of $f$ with respect to the lower walk $N_k^i$ can be written as an expectation of the (non)-expansion of $f$ restricted to links. That is for all $i \leq k$ and $j \leq k-i$:
\begin{equation}
\bar{\Phi}_{N^i_k}(f) = \underset{\tau \sim \pi^f_{j}}{\mathbb{E}}[\bar{\Phi}_{N^{i}_{k-j}|_{\tau}}(f|_\tau)],
\end{equation}
where $\pi^f_{j}$ is the distribution given by:
\[
\pi^f_{j}(\tau) = \pi_j(\tau) \frac{\langle f|_\tau, f|_\tau \rangle}{\langle f,f \rangle}.
\]
\end{lemma}
\begin{proof}
Localizing via \Cref{eq:gar-2}, we have:
\begin{align*}
    \bar{\Phi}(f) &= \frac{1}{\langle f,f \rangle}\sum\limits_{\tau \in X(j)} \pi_j(\tau) \langle f|_\tau, N_{k-j}^i|_\tau f|_\tau \rangle\\
    &= \frac{1}{\langle f, f \rangle}\sum\limits_{\tau \in X(j)} \pi_j(\tau) \langle f|_\tau, f|_\tau \rangle \bar{\Phi}_{N^{i}_{k-j}|_{\tau}}(f|_\tau)\\
    &=\sum\limits_{\tau \in X(j)}\pi_j^f(\tau)\bar{\Phi}_{N^{i}_{k-j}|_{\tau}}(f|_\tau).
\end{align*}
The fact that $\pi_j^f$ is a distribution follow from noting that $\pi_j^f>0$ by construction, and by a simple application of \Cref{eq:gar-1}:
\begin{align*}
    \sum\limits_{\tau \in X(j)} \pi_j^f(\tau) &= \sum\limits_{\tau \in X(j)} \pi_j(\tau) \frac{\langle f|_\tau, f|_\tau \rangle}{\langle f,f \rangle}\\
    &= \frac{1}{\langle f,f \rangle}\underset{\tau \sim \pi_j}{\mathbb{E}}[\langle f|_\tau, f|_\tau \rangle]\\
    &= 1
\end{align*}
\end{proof}
We emphasize that this variant of Garland's lemma differs from the typical form in that the distribution over links is \textit{function dependent}. Thus slightly more care must be taken in its application, as typical terms which average out over the standard link distribution may not do so in the function dependent one.
\subsection{The KO/GK Decomposition}
The proof of \Cref{thm:expand} also relies on a decomposition of Gotlib and Kaufman \cite{gotlib2022fine}, itself based on the following decomposition of Kaufman and Oppenheim \cite{kaufman2020high}, which we restate here for convenience.
\begin{definition}[KO-Decomposition \cite{kaufman2020high}]
Let $(X,\Pi)$ be a simplicial complex, $0 \leq k \leq d$, and $f \in C_k$. There exists an orthogonal decomposition of $f$ into levels
\[
f = \sum\limits_{i=0}^k f_i
\]
such that $f_i \in Im(U^k_i) \cap Ker(D^k_{i-1})$.
\end{definition}
The KO-basis is not known to be an eigenbasis (even approximately) on weak HDX, but Gotlib and Kaufman \cite{gotlib2022fine} recently showed it is still possible to gain some advantage by understanding how a function projects onto its components. We will use a variant of their result for the lower walk.
\begin{theorem}\label{thm:decomp}
For any simplicial complex $(X,\Pi)$, $0 < k \leq d$, and $f \in C_{k}$, we have:
\[
\langle N_k^1 f,f \rangle \leq \mathbb{E}[f]^2 + \sum\limits_{i=1}^{k-1}\left(1-\frac{1}{k-i+1}\prod\limits_{j=i-1}^{k-2}(1-\gamma_j) \right) \langle f,f_i \rangle.
\]
\end{theorem}
The proof is similar to \cite[Theorem 7.9]{gotlib2022fine}, which focuses on the non-lazy upper walk instead, and is given in \Cref{app:GK}.

\subsection{A Level-$i$ Inequality and the Upper Bound}\label{sec:pseudo}
With this in mind, a natural goal is now to show that global functions project mostly onto large components, and therefore expand well. In fact, it turns out a variant of this statement is true for \textit{all} simplicial complexes, where the strength of the bound scales with the smallest non-zero eigenvalue of the upper walks.\footnote{Note these walks are PSD by adjointness of $D$ and $U$, so $\lambda^{\cancel{0}}_{\text{min}}$ is strictly positive.}
\begin{theorem}[Level-$i$ Inequality]\label{thm:i-inequality}
Let $(X,\Pi)$ be a simplicial complex. Then for any $0 < k \leq d$ and $(\varepsilon,i)$-global boolean function $f \in C_k$, the following level-$i$ inequality holds:
\[
\sum\limits_{j=1}^i \langle f, f_j \rangle \leq \frac{1}{\lambda^{\cancel{0}}_{\text{min}}(\widehat{N}_i^{k-i})} \varepsilon\mathbb{E}[f],
\]
where $\lambda^{\cancel{0}}_{\text{min}}$ is the smallest non-zero eigenvalue.
\end{theorem}
Unfortunately, while \Cref{thm:i-inequality} is a non-trivial statement for any fixed complex, $\lambda^{\cancel{0}}_{\text{min}}(\widehat{N}_i^{k-i})$ may be unbounded, i.e., may depend on the number of vertices in the complex). The crucial observation is that this issue can be fixed whenever the complex is locally far from bipartite.

We focus in particular on the special case of $i=1$, which is the only setting we need for our approach. 
\begin{corollary}[Level-1 Inequality]\label{cor:level-1}
Let $(X,\Pi)$ be a simplicial complex, $k \leq d$, and $f \in C_k$ any function on $k$-faces. Then
\[
\langle f,f_1 \rangle \leq \frac{k}{1 + \gamma_{0}^{(-)}(k-1)}\varepsilon\mathbb{E}[f] \leq \frac{k}{1 + \gamma_{d-2}^{(-)}\cdot\frac{k-1}{1-(d-2)\gamma_{d-2}^{(-)}}}\varepsilon\mathbb{E}[f]
\]
\end{corollary}
\begin{proof}
The second inequality follows from the Trickling-Down Theorem. Then by \Cref{thm:i-inequality}, it is enough to lower bound $\lambda_{\text{min}}(D^k_1U_1^k) \leq \lambda^{\cancel{0}}_{\text{min}}(D^k_1U_1^k)$, i.e.\ to show that for any $f \in C_k$:
\[
\frac{\langle f, D^k_1U^k_1f\rangle}{\langle f,f \rangle} \geq \left(\frac{1}{k} + \gamma_0^{(-)}\frac{k-1}{k}\right).
\]
We appeal to the fact that $D^k_1U^k_1 = \frac{k-1}{k}M_1^+ + \frac{1}{k}I$ (see e.g.\ \cite{alev2019approximating,gotlib2022fine}) and observe that for any $f \in C_k$:
\begin{align*}
    \langle f, D^k_1U^k_1f \rangle &= \frac{1}{k} \langle f,f \rangle + \frac{k-1}{k}\langle f, M_1^+ f \rangle\\
    &\geq \left(\frac{1}{k} + \gamma_0^{(-)}\frac{k-1}{k}\right)\langle f,f \rangle
\end{align*}
as desired.
\end{proof}


We now prove \Cref{thm:i-inequality} itself. Similar to the approach of \cite{bafna2020high}, we first show a level-$i$ inequality in terms of variance before reducing to the standard setting.
\begin{lemma}[Level-$i$ Inequality (Variance)]\label{lem:variance-nonzero}
Let $(X,\Pi)$ be a simplicial complex, $k \leq d$, and $f: C_k \to \mathbb{R}$ any function on $k$-faces. Then for all $i > 0$:
\[
\text{Var}(D^k_i f) \geq \lambda^{\cancel{0}}_{\text{min}}(\widehat{N}_i^{k-i}) \sum\limits_{j=1}^i \langle f,f_j \rangle
\]
\end{lemma}
\begin{proof}
First, observe that by adjointness of the averaging operators we can write:
\begin{align*}
\langle D^k_i f, D^k_i f \rangle &= \langle f, U^k_iD^k_i f \rangle\\
&= \sum\limits_{j=0}^k \langle f, U^k_iD^k_i f_j \rangle\\
&= \sum\limits_{j=0}^i \langle f, U^k_iD^k_i f_j \rangle
\end{align*}
since $f_j \in \text{Ker}(D^k_{j-1})$, and therefore that 
\[
\text{Var}(D^k_i f) = \sum\limits_{j=1}^i \langle f, U^k_iD^k_i f_j \rangle
\]
since $\langle f, f_0\rangle =\mathbb{E}[f]^2= \mathbb{E}[D^k_0f]^2$. 

The trick is now to observe that each $f_j$ is completely orthogonal to the kernel of $U^k_iD^k_i$. If this is the case, we then have
\begin{align*}
\text{Var}(D^k_i f) &\geq \lambda_{\text{min}}^{\cancel{0}}(U^k_iD^k_i) \sum\limits_{j=1}^i \langle f, f_j \rangle\\
&= \lambda^{\cancel{0}}_{\text{min}}(\widehat{N}_i^{k-i}) \sum\limits_{j=1}^i \langle f,f_j \rangle
\end{align*}
since $U^k_iD^k_i$ and $D^k_iU^k_i$ are PSD and have the same non-zero spectrum.

To see that $f_j$ has no projection onto $\text{Ker}(U^k_iD^k_i)$, observe that $\text{Ker}(U^k_iD^k_i) = \text{Ker}(D^k_i)$ by adjointness of $D$ and $U$. Namely if $U^k_iD^k_if=0$, then $\langle f, U^k_iD^k_if \rangle = \norm{D^k_if}^2=0$ so $f \in \text{Ker}(D^k_i)$ (the inclusion in the other direction is trivial). Finally since $j \leq i$ we have
\[
f_j \in \text{Im}(U_j^k) \supset \text{Im}(U_i^k) = \text{Ker}(D^k_i)^\perp
\]
by definition of $f_j$ and adjointness of $D$ and $U$, so we are done.
\end{proof}
The proof of \Cref{thm:i-inequality} now follows immediately from the $\ell_\infty$-to-$\ell_2$ reduction introduced in \cite{bafna2020high}.
\begin{proof}[Proof of \Cref{thm:i-inequality}]
BHKL observed that any $(\varepsilon,i)$-global boolean function satisfies $\text{Var}(D^k_i f) \leq \varepsilon\mathbb{E}[f]$. Combined with \Cref{lem:variance-nonzero} this completes the proof.
\end{proof}
We note that this result can be generalized to real-valued functions under some light regularity conditions (and arbitrary non-negative functions) as in \cite{bafna2020high}, but we focus on the boolean case here for simplicity.

We are finally equipped to prove \Cref{thm:expand}.

\begin{proof}[Proof of \Cref{thm:expand}]
For notational convenience, for $1 \leq \ell < k \leq d$ define
\[
\lambda_{k,\ell} \coloneqq 1-\frac{1}{k-\ell+1}\prod\limits_{j=\ell-1}^{k-2}(1-\gamma_j)
\]
and set $\lambda_{k,0}=1$ and $\lambda_{k,k}=0$.

We first prove the result for $i=1$, then show how to reduce to this setting in general via localization. When $i=1$, we actually show a slightly stronger bound
\begin{equation}
    (1-\mathbb{E}[f])(1-\lambda_{k,2}) - \frac{(\lambda_{k,1}-\lambda_{k,2})}{\frac{1}{k} + \frac{k-1}{k}\gamma^{(-)}_0}\varepsilon.
\end{equation}
This follows fairly immediately from \Cref{thm:decomp} and properties of the $\{\lambda_{k,i}\}$ and $\{f_i\}$ decomposition:
\begin{align*}
\Phi(f) &\geq 1 - \frac{1}{\mathbb{E}[f]}\sum\limits_{\ell=0}^{k}\lambda_{k,\ell} \langle f,f_\ell \rangle\\
&\geq 1 - \frac{1}{\mathbb{E}[f]}\sum\limits_{\ell=0}^{1}\lambda_{k,\ell}\langle f,f_\ell \rangle - \frac{\lambda_{k,2}}{\mathbb{E}[f]}\sum\limits_{\ell=2}^{k} \langle f,f_\ell \rangle & \text{($\lambda_{k,i}$ are non-increasing)}\\
&\geq 1 - \frac{1}{\mathbb{E}[f]}\sum\limits_{\ell=0}^{1}\lambda_{k,\ell}\langle f,f_\ell \rangle - \frac{\lambda_{k,2}}{\mathbb{E}[f]}\left(\mathbb{E}[f] - \sum\limits_{\ell=0}^{1}\langle f,f_\ell \rangle\right) & \text{(Booleanity)}\\
&= 1 - \frac{1}{\mathbb{E}[f]}\sum\limits_{\ell=0}^{1}(\lambda_{k,\ell}-\lambda_{k,2})\langle f,f_\ell \rangle -\lambda_{k,2}\\
&= (1-\mathbb{E}[f])(1-\lambda_{k,2}) - \frac{1}{\mathbb{E}[f]}(\lambda_{k,1}-\lambda_{k,2})\langle f,f_1\rangle & \text{($\langle f,f_0\rangle = \mathbb{E}[f]^2$)}\\
&\geq (1-\mathbb{E}[f])(1-\lambda_{k,2}) - \frac{(\lambda_{k,1}-\lambda_{k,2})}{\frac{1}{k} + \frac{k-1}{k}\gamma^{(-)}_0}\varepsilon & \text{(\Cref{cor:level-1})}
\end{align*}
We observe by the same proof, the above also holds even when $\lambda_{k,1}$ and $\lambda_{k,2}$ are replaced with any values $a$,$b$ such that $a \geq \lambda_{k,1}$ and $b \geq \lambda_{k,2}$.

The inductive step relies on the observation that global functions are approximately closed under localization. This requires some care to achieve the stated bound. In particular, observe that if $f$ is $(\varepsilon,i)$-global, for any $\tau \in X(i-1)$ we can write
\[
\mathbb{E}[f|_\tau] = \mathbb{E}[f]+\delta_\tau
\]
where $|\delta_\tau| \leq \varepsilon$, and further that $f|_\tau$ is then $(\varepsilon-\delta_\tau,1)$-global. With this in mind, applying Garland's method for expansion (\Cref{lemma:local-to-global-expansion}), we have
\[
 \Phi(f) = \underset{\tau \sim \pi_{i-1}^f}{\mathbb{E}}[\Phi_{\tau}(f|_\tau)].
\]
Observe that for all $\tau \in X(i-1)$, the `local' $\lambda_{i,j}$ values in the link $X_\tau$ (denoted $\lambda^\tau_{i,j}$ and defined analogously with respect to $X_\tau$) are upper bounded by their corresponding values in the original complex, that is $\lambda^\tau_{k-i+1,1} \leq \lambda_{k,i}$ and $\lambda^\tau_{k-i+1,2} \leq \lambda_{k,i+1}$. Then applying the base case to the above localization, we get
\begin{align*}
    \Phi(f)
    &\geq \underset{\tau \sim \pi_{i-1}^f}{\mathbb{E}}\left[(1-\mathbb{E}[f]-\delta_\tau)(1-\lambda_{k,i+1}) - \frac{(\lambda_{k,i}-\lambda_{k,i+1})}{\frac{1}{k-i+1} + \frac{k-i}{k-i+1}\gamma^{(-)}_{i-1}}(\varepsilon-\delta_\tau)\right]\\
    &=(1-\mathbb{E}[f])(1-\lambda_{k,i+1}) - \frac{(\lambda_{k,i}-\lambda_{k,i+1})}{\frac{1}{k-i+1} + \frac{k-i}{k-i+1}\gamma^{(-)}_{i-1}}\varepsilon + \left(\frac{(\lambda_{k,i}-\lambda_{k,i+1})}{\frac{1}{k-i+1} + \frac{k-i}{k-i+1}\gamma^{(-)}_{i-1}} - (1-\lambda_{k,i+1}) \right)\underset{\tau \sim \pi_{i-1}^f}{\mathbb{E}}\left[\delta_\tau \right]\\
    &\geq (1-\mathbb{E}[f])(1-\lambda_{k,i+1}) - \frac{(\lambda_{k,i}-\lambda_{k,i+1})}{\frac{1}{k-i+1} + \frac{k-i}{k-i+1}\gamma^{(-)}_{i-1}}\varepsilon
\end{align*}
where in the final inequality we have used the fact that
\[
\left(\frac{(\lambda_{k,i}-\lambda_{k,i+1})}{\frac{1}{k-i+1} + \frac{k-i}{k-i+1}\gamma^{(-)}_{i-1}} - (1-\lambda_{k,i+1}) \right) \geq 0.
\]
This in turn follows from noting that $\gamma^{(-)}_{i-1} \leq 0$, and
\begin{align*}
    (k-i+1)(\lambda_{k,i}-\lambda_{k,i+1}) &= (k-i+1)\left(\frac{1}{k-i} - \frac{1}{k-i+1}(1-\gamma_{i-1}) \right)\prod\limits_{j=i}^{k-2}(1-\gamma_j)\\
    &=\left(\frac{1}{k-i} +\gamma_{i-1} \right)\prod\limits_{j=i}^{k-2}(1-\gamma_j)\\
    &\geq \frac{1}{k-i}\prod\limits_{j=i}^{k-2}(1-\gamma_j)\\
    &= 1-\lambda_{k,i+1}.
\end{align*}
\end{proof}

\section{A Local-to-Global Small-Set Expansion Theorem}
In the previous section, we analyzed the structure of non-expanding functions with respect to the down-up walk. A second key operator of interest in application is the noise operator, where the analogous bounds on e.g. the cube lead to classical `small-set expansion' theorems. In this section, we show a `local-to-global' theorem for (global) small-set expansion: any complex whose links satisfy such a notion must itself be a global small-set expander. Later, we'll see this gives as an application the first small-set expansion theorem for (non-partite) Ramanujan complexes of \cite{lubotzky2005explicit}, a regime in which prior Fourier analytic techniques fail.
\subsection{Localizing the Noise Operator}
The first step in our proof is to generalize \Cref{lemma:local-to-global-expansion} to the noise operator. To this end, let $\mathcal{B}_{n}^p$ denote the standard mean $p$ binomial distribution on $n$ trials, and let $d_{TV}$ denote the \textit{Total Variation Distance}, that is the distance measure between distributions $D$ and $D'$ over universe $\Omega$ defined by:
\[
d_{TV}(D,D') \coloneqq \max_{E \subseteq \Omega}|D(E) - D'(E)| = \frac{1}{2}\sum\limits_{x\in \Omega} |D(x) - D'(x)|
\]
We first argue the noise operator can be localized up to an error term that  scales with the TV-distance between adjacent binomial distributions.
\begin{proposition}[Garland's Lemma for Noise Operators]\label{cor:noisy-garland1}
Let $(X,\Pi)$ be a weighted simplicial complex. Let $\rho \in [0,1]$ and $j,k \in \mathbb{N}$. The expansion of $T_\rho$ can be approximately localized over $j$-links:
\[
\left |\bar{\Phi}_{T^k_\rho}(f) - \mathbb{E}_{\tau \sim \pi_j^f}[\bar{\Phi}_{T^{k-j}_\rho|_\tau}(f|_\tau)]\right | \leq 2d_{TV}(\mathcal{B}^\rho_{k},\mathcal{B}^\rho_{k-j}).
\]
\end{proposition}
\begin{proof}
Recall that the noise operator is given by a binomially-distributed convex combination of lower walks:
    \[
T^k_\rho \coloneqq \sum\limits_{i=0}^k \mathcal{B}^\rho_k(i)N_k^{k-i},
\]
Naively, one might wish to apply \Cref{lemma:local-to-global-expansion} directly to move to the localized noise operators, but this has two issues. First, we can only apply \Cref{lemma:local-to-global-expansion} whenever $i \leq k-j$. Second even accounting for this fact, the binomial coefficients need to be shifted to dimension $k-j$ as well. To account for this, we introduce the \textit{shifted} noise operator:
    \[
T^{k,j\downarrow}_\rho \coloneqq \sum\limits_{i=0}^{k-j} \mathcal{B}_{k-j}^\rho(i) N_k^{k-i}.
\]
By \Cref{lemma:local-to-global-expansion}, the shifted noise operator localizes to the desired form:
\begin{align*}
    \bar{\Phi}_{T^{k,j\downarrow}_\rho}(f) &= \sum\limits_{i=0}^{k-j}  \mathcal{B}_{k-j}^\rho(i) \bar{\Phi}_{N_k^{k-i}}(f)\\
    &=\sum\limits_{i=0}^{k-j}  \mathcal{B}_{k-j}^\rho(i) \underset{\tau \sim \pi_j^f}{\mathbb{E}}[\bar{\Phi}_{N_{k-j}^{k-i}|_\tau}(f|_\tau)]\\
    &= \underset{\tau \sim \pi_j^f}{\mathbb{E}}[\bar{\Phi}_{T_\rho^{k-j}|_\tau}(f|_\tau)]
\end{align*}
where the final step is by linearity of inner products. Thus it is enough to show that
\[
|\bar{\Phi}_{T^{k}_\rho}(f) -\bar{\Phi}_{T^{k,j\downarrow}_\rho}(f)| \leq 2d_{TV}(\mathcal{B}_k^\rho,\mathcal{B}_{k-j}^\rho).
\]
This is immediate from the fact that $\bar{\Phi}_{N_k^i}(f) \in [0,1]$, since by the triangle inequality:
\begin{align*}
    |\bar{\Phi}_{T^{k}_\rho}(f) -\bar{\Phi}_{T^{k,j\downarrow}_\rho}(f)| &\leq \sum\limits_{i=0}^k |B_k^\rho(i) - B_{k-j}^\rho(i)|\bar{\Phi}_{N_k^{k-i}}(f)\\
    &\leq \sum\limits_{i=0}^k|B_k^\rho(i) - B_{k-j}^\rho(i)|\\
    &= 2d_{TV}(\mathcal{B}_k^\rho,\mathcal{B}_{k-j}^\rho)
\end{align*}
as desired.
\end{proof}
For fixed $j$, the additive error term in \Cref{cor:noisy-garland1} goes to $0$ at a rate of about $\frac{1}{\sqrt{k}}$. While this is sufficient for most use cases, it only recovers small-set expansion for sets of size at least $\text{poly}(k^{-1})$, while one might hope to show such a bound for sets of size up to $\text{exp}(-k)$. We show this is possible if the links are additionally good high dimensional expanders.
\begin{proposition}[Refined Garland's Lemma for Noise Operators on HDX]\label{cor:noisy-garland2}
Let $(X,\Pi)$ be a $k$-dimensional weighted simplicial complex such that every $j$-link is a $\gamma$-two-sided or partite one-sided HDX. Then for any $\rho \in [0,1-1/e]$ and $j \leq d$, the expansion of $T_\rho$ can be approximately localized over $j$-links:
\[
\bar{\Phi}_{T^k_\rho}(f) \leq \mathbb{E}_{\tau \sim \pi_j^f}[\bar{\Phi}_{T^{k-j}_{\rho}|_\tau}(f|_\tau)] + F_{\mathcal{B}_{k}^{\rho}}(j-1) + c_k\gamma
\]
where $c_k \leq 2^{O(k)}$ and $F_{\mathcal{B}_{k}^{\rho}}(\cdot)$ denotes the CDF of $\mathcal{B}_{k}^{\rho}$.
\end{proposition}
For constant $j$ and $\gamma$ sufficiently small, both error terms are inverse exponential in $k$. Such error is negligible since the noise operator \textit{always} has non-expansion at least inverse exponential in $k$ anyway.

The proof of \Cref{cor:noisy-garland2} combines the basic localization of \Cref{lemma:local-to-global-expansion} with the Fourier analytic machinery of \cite{dikstein2018boolean,bafna2021hypercontractivity,gur2021hypercontractivity}. To this end, we first introduce a basic version of the latter.
\begin{theorem}[Fourier Analysis on HDX {\cite{dikstein2018boolean,bafna2021hypercontractivity,gur2021hypercontractivity}}]\label{thm:Fourier-Analysis}
    Let $(X,\Pi)$ be a $k$-dimensional $\gamma$-two-sided or partite one-sided HDX. Then there exists a decomposition $f=\sum\limits_{i=0}^k f_i$ such that:
    \[
    \bar{\Phi}_{N^{k-i}_k}(f) = \sum\limits_{\ell=0}^i \frac{{k - \ell \choose i - \ell}}{{k \choose i}} \frac{\langle f_\ell, f_\ell \rangle}{\langle f, f \rangle} \pm c_k\gamma
    \]
    where $c_k \leq 2^{O(k)}$. Similarly:    
    \[
    \bar{\Phi}_{T_\rho^k}(f) = \sum\limits_{\ell=0}^k \rho^\ell \frac{\langle f_\ell, f_\ell \rangle}{\langle f, f \rangle} \pm c_k\gamma.
    \]
\end{theorem}
We note that \Cref{thm:Fourier-Analysis} does not appear anywhere in the literature, but follows without too much difficulty from techniques developed in \cite{dikstein2018boolean,bafna2021hypercontractivity,gur2021hypercontractivity}. We give the proof for completeness in \Cref{app:basis}. We can now prove \Cref{cor:noisy-garland2}.
\begin{proof}[Proof of \Cref{cor:noisy-garland2}]
    At a high level, the proof is similar to \Cref{cor:noisy-garland1}, but we will separate the `shifting' of $T_\rho$ into two parts. The first, which will introduce our additive error, is to truncate the terms for $i \leq j$ which cannot be localized via \Cref{lemma:local-to-global-expansion}:
    \begin{align*}
    \bar{\Phi}_{T^k_\rho}(f) &= \sum\limits_{i=j}^{k} \mathcal{B}_{k}^\rho(i)\bar{\Phi}_{N^{k-i}_k}(f)  + \sum\limits_{i=0}^{j-1} \mathcal{B}_{k}^\rho(i)\bar{\Phi}_{N^{k-i}_k}(f)\\
    &\leq \sum\limits_{i=j}^{k} \mathcal{B}_{k}^\rho(i)\bar{\Phi}_{N^{k-i}_k}(f) + F_{\mathcal{B}_k^\rho}(j-1)
    \end{align*}
    once again using the fact that $\bar{\Phi}_{N_k^j} \in [0,1]$. We can apply now \Cref{lemma:local-to-global-expansion} to the left-hand term directly:
    \[
    \sum\limits_{i=j}^{k} \mathcal{B}_{k}^\rho(i)\bar{\Phi}_{N^{k-i}_k}(f) = \underset{\tau \sim \pi^f_j}{\mathbb{E}}\left[\sum\limits_{i=j}^{k} \mathcal{B}_{k}^\rho(i)\bar{\Phi}_{N^{k-i}_{k-j}|_\tau}(f|_\tau)\right].
    \]
    The second step is to give a finer-grained analysis of the inner expectation for each $\tau \in X(j)$ (rather than simply shifting $\mathcal{B}_k^\rho$ and bounding by TV-distance). In particular, we'll argue that for every $\tau \in X(j)$:
\begin{equation*}
    \sum\limits_{i=j}^{k} \mathcal{B}_{k}^\rho(i)\bar{\Phi}_{N^{k-i}_{k-j}|_\tau}(f|_\tau) \leq \bar{\Phi}_{T_{\rho}^{k-j}|_\tau}(f|_\tau) +  c_k\gamma
\end{equation*}
which completes the proof. To see this, first expand the LHS in terms of its Fourier basis:
\begin{align*}
\sum\limits_{i=j}^{k} \mathcal{B}_{k}^\rho(i)\bar{\Phi}_{N^{k-i}_{k-j}|_\tau}(f|_\tau) &\leq \frac{1}{\langle f|_\tau, f|_\tau \rangle}\sum\limits_{i=0}^{k-j} \mathcal{B}_{k}^\rho(i+j)\sum\limits_{\ell=0}^i \frac{{k-j-\ell\choose i -\ell}}{{k  - j \choose i}}\langle (f|_\tau)_\ell,(f|_\tau)_\ell\rangle\\
&= \frac{1}{\langle f|_\tau, f|_\tau \rangle}\sum\limits_{\ell=0}^{k-j}\left(\sum_{i=\ell}^{k-j}{k \choose i+j}\frac{{k-j-\ell \choose i-\ell}}{{k-j \choose i}}\rho^{i+j}(1-\rho)^{k-j-i}\right)\langle (f|_\tau)_\ell,(f|_\tau)_\ell\rangle.
\end{align*}
The key is now to observe that the inner sum can be bounded by the (approximate) eigenvalues of $T_\rho$, namely that the following binomial inequality holds.
\begin{claim}\label{claim:binomial}
For all $k \in \mathbb{N}$, $j \leq k$, and $\ell \leq k-j$:
    \[
    \sum_{i=\ell}^{k-j}{k \choose i+j}\frac{{k-j-\ell \choose i-\ell}}{{k-j \choose i}}\rho^{i+j}(1-\rho)^{k-j-i} \leq \rho^\ell.
    \]
\end{claim}
Then re-applying \Cref{thm:Fourier-Analysis} we'd have:
\begin{align*}
    \sum\limits_{i=j}^{k} \mathcal{B}_{k}^\rho(i)\bar{\Phi}_{N^{k-i}_{k-j}|_\tau}(f|_\tau) &\leq \frac{1}{\langle f|_\tau, f|_\tau \rangle}\sum\limits_{\ell=0}^{k-j}\rho^\ell \langle(f|_\tau)_\ell,(f|_\tau)_\ell\rangle\\
&\leq \bar{\Phi}_{T_\rho^{k-j}|_\tau}(f|_\tau) + c_k\gamma
\end{align*}
as desired.

We leave the proof of \Cref{claim:binomial}, which follows from fairly elementary binomial manipulation, to \Cref{app:basis}.
\end{proof}
\subsection{The Small-Set Expansion Theorem}
We now give the general statement of our local-to-global small-set expansion theorem. We call a complex $j$-locally SSE if its links satisfy an SSE Theorem for strongly global functions.
\begin{definition}[Locally SSE Complexes]
    Let $(X,\Pi)$ be a simplicial complex. We say $(X,\Pi)$ is $j$-locally $\phi$-SSE if for every $i \leq d-j$, $\varepsilon>0$, $j$-link $X_\tau$, and $(\varepsilon,i-j)$-strongly global function $f \in C_{d-j}(X_\tau)$ the expansion of $h$ is at least:
\[
\Phi_{T_{\rho}^{d-j}|_\tau}(f|_\tau) \geq \phi(\varepsilon,i)
\]
\end{definition}
It is an almost immediate corollary of \Cref{cor:noisy-garland1,cor:noisy-garland2} that global sets on locally SSE complexes expand.
\begin{corollary}
    \label{thm:local-to-global}
Let $(X,\Pi)$ be a $j$-locally $\phi$-SSE simplicial complex. Then the expansion of any $(\varepsilon,i)$-strongly global function $f \in C_k$ of density at most $\varepsilon$ satisfies:
\[
\Phi_{T^k_\rho}(f) \geq \phi(\varepsilon,i) - 2d_{TV}(\mathcal{B}^\rho_{k},\mathcal{B}^\rho_{k-j}).
\]
Moreover, if the $j$-links are $\gamma$-two-sided or partite one-sided HDX then
\[
\Phi_{T^k_\rho}(f) \geq \phi(\varepsilon,i) - F_{\mathcal{B}_{k}^{\rho}}(j-1) - c_k\gamma
\]
where $c_k \leq 2^{O(k)}$. 
\end{corollary}
\begin{proof}
This is essentially immediate from combining \Cref{cor:noisy-garland1,cor:noisy-garland2} with the observation that any $(i,\varepsilon)$-strongly global function is $(i-j,\varepsilon)$-strongly global upon localization to any $j$-link. We therefore have:
\begin{align*}
    \Phi_{T^k_\rho}(f) \geq \underset{\tau \sim \pi^f_{j}}{\mathbb{E}}[\Phi_{{T^{k-j}_\rho}|_\tau}(f|_\tau)] - 2d_{TV}(\mathcal{B}^\rho_{k},\mathcal{B}^\rho_{k-j}) \geq \phi(\varepsilon,i)- 2d_{TV}(\mathcal{B}^\rho_{k},\mathcal{B}^\rho_{k-j})
\end{align*}
in the general case and
\[
\Phi_{T^k_\rho}(f) \geq \underset{\tau \sim \pi^f_{j}}{\mathbb{E}}[\Phi_{{T^{k-j}_\rho}|_\tau}(f|_\tau)] - F_{\mathcal{B}_{k}^{\rho}}(j-1) - c_k\gamma \geq \phi(\varepsilon,i) - F_{\mathcal{B}_{k}^{\rho}}(j-1) - c_k\gamma
\]
when the $j$-links are $\gamma$-HDX as desired.
\end{proof}


\section{Applications}\label{sec:app}
In this section we give applications of our framework to the behavior of non-expanding sets on several families of well-studied complexes. In particular, we give a new characterization of low influence functions on clique and product complexes, and a small-set expansion theorem for the seminal Ramanujan complexes of \cite{lubotzky2005explicit}.
\subsection{Low Influence Functions on Weak HDX}
Bafna, Hopkins, Kaufman, and Lovett \cite{bafna2020high,bafna2021hypercontractivity}, and independently Gur, Lifshitz, and Liu \cite{gur2021hypercontractivity} characterized the structure of non-expanding sets on near-perfect local-spectral expanders. Unfortunately, very few families of HDX are known to satisfy such strong requirements, and those that do are highly algebraic in nature. Applying our framework for weak expansion, we show such characterizations also holds on more `everyday' objects, albeit in a quantitatively weaker sense. To illustrate this fact, we'll first take a closer look at a classical combinatorial setting: clique-complexes.
\begin{definition}[Clique-complex]
Given a graph $G=(V,E)$, the $k$-dimensional clique-complex $K_{G,k}$ is the simplicial complex induced by the uniform distribution over $k$-cliques of $G$.
\end{definition}
We show global sets expand on clique-complexes of dense graphs.
\begin{theorem}[Expansion in Clique-complexes]\label{thm:clique}
Fix $k \in \mathbb{N}$ and let $G=(V,E)$ be any graph with minimum degree at least $\Delta_{\text{min}} \geq \frac{k-2}{k-1}|V| + \frac{k||A_{\bar{G}}||-1}{k-1}$, where $A_{\bar{G}}$ denotes the un-normalized adjacency matrix of $G$'s complement. Then the expansion of any boolean $(\varepsilon,i)$-global function $f$ is at least:
\[
\Phi(f) \geq (1-\mathbb{E}[f])\frac{i+1}{k(k-i)} - \frac{5}{4}\varepsilon.
\]
\end{theorem}
\begin{proof}
It will be convenient to phrase the result in the equivalent setup of $k$-size independent sets, denoted $I_{G,k}$. In particular, it is enough to prove the result holds over $I_{G,k}$ so long as
\begin{equation}\label{eq:degree}
    \Delta_{\text{max}} \leq \frac{|V|-k||A_G||+1}{k-1}
\end{equation}
as we can then apply this to the independent set complex of the complement of $G$ (which is exactly $K_{G,k}$).

In their celebrated work on the spectral gap of high order random walks, Alev and Lau \cite{alev2020improved} showed that $I_{G,k}$ satisfies $\gamma_{k-2} \leq \frac{1}{k}$ as long as\footnote{This is not exactly the statement given in \cite[Lemma 4.3]{alev2020improved}, but it is immediate from the proof solving for dependence on $\Delta_{\text{max}}$ instead of $k$.} 
\[
\Delta_{\text{max}} \leq \frac{|V|-k|\lambda_{min}(A_G)|+1}{k-1}.
\]
Note that this always holds for any graph satisfying \Cref{eq:degree}.

Thus to apply our characterization theorem, it is enough to bound the negative local spectra. We will do this by proving that under \Cref{eq:degree}, the co-dimension $2$ links of $I_{G,k}$ also satisfy $\gamma^{(-)}_{k-2} \geq -\frac{1}{k}$. Then by Oppenheim's Trickling-Down Theorem we have 
\[
\gamma_{i}^{(-)} \geq -\frac{1}{2k-2-i}
\]
and the result follows immediately from plugging this bound into \Cref{thm:expand}.

Our bound on $\gamma^{(-)}_{k-2}$ follows largely the same strategy as Alev and Lau's original analysis. Let $G_S$ denote the link of a co-dimension $2$ face $S \in X(k-2)$, and let $N[S]$ denote the union of $S$ and vertices of $G$ with a neighbor in $S$. As in Alev-Lau, the idea is now to observe that since the top level faces are distributed uniformly, the random walk matrix underlying $M_S$ can be written as
\[
M_S = D_S^{-1}(J-I-A_{G[\overline{N[S]}]}),
\]
where $A_{G[\overline{N[S]}]}$ is the un-normalized adjacency matrix of the induced graph on $G[\overline{N[S]}]$ and $D_S$ is the standard diagonal degree matrix of $G_S$. We can then lower bound the smallest eigenvalue of $M_S$ by upper bounding the largest eigenvalue of $-M_S$ as:
\begin{align*}
    \lambda_1(-M_S) &\leq \lambda_1(D_S^{-1/2}(A_{G[\overline{N[S]}]} + I)D_S^{-1/2}) - \lambda_1(D_S^{-1}J)\\
    &=  \lambda_1(D_S^{-1/2}(A_{G[\overline{N[S]}]} + I)D_S^{-1/2})- \sum\limits_{v \in G_S}\frac{1}{\text{deg}_{G_S}(v)}\\
    &\leq  \norm{D_S^{-1}} \lambda_1(A_{G[\overline{N[S]}]} + I) - \sum\limits_{v \in G_S}\frac{1}{\text{deg}_{G_S}(v)}& \text{(Variational Characterization)}\\
    &=  \norm{D_S^{-1}} (\lambda_1(A_{G[\overline{N[S]}]})+1)- \sum\limits_{v \in G_S}\frac{1}{\text{deg}_{G_S}(v)}\\
    &\leq \norm{D_S^{-1}} (\lambda_1(A_{G})+1)  - \sum\limits_{v \in G_S}\frac{1}{\text{deg}_{G_S}(v)}& \text{(Cauchy's Interlacing Theorem).}\\
    &\leq \norm{D_S^{-1}} (\lambda_1(A_{G})-1)
\end{align*}
where the last step follows from observing that $\norm{D_S^{-1}}=\frac{1}{\text{min-deg}(G_S)}$ and $\sum\limits_{v \in G_S}\frac{1}{\text{deg}_{G_S}(v)} \geq \frac{2}{\text{min-deg}(G_S)}$.

Finally as observed in \cite[Lemma 4.3]{alev2020improved}, note that we can bound the degree of any $v \in G_S$ by
\[
\text{deg}_{G_S}(v) = |V| - |N[S]|-(\text{deg}_{G[\overline{N[S]}]}(v)+1) \geq |V|-(\Delta_{\text{max}}+1)(k-1).
\]
Thus we have
\[
\norm{D_S^{-1}} \leq \frac{1}{|V| - (\Delta_{\text{max}}+1)(k-1)},
\]
and altogether
\[
|\lambda_{\text{min}}(M_S)| \leq \frac{\lambda_1(A_G)-1}{|V| - (\Delta_{\text{max}}+1)(k-1)}.
\]
Combined with Alev and Lau's bound on $\lambda_2(M_S)$ and our assumption on degree, this gives
\[
\max\{\lambda_2(M_S),|\lambda_{\text{min}}(M_S)|\} \leq \frac{\norm{A_G}-1}{|V| - (\Delta_{\text{max}}+1)(k-1)} \leq \frac{1}{k}
\]
as desired.
\end{proof}
Since $\norm{A_{\bar{G}}} \leq \Delta_{\text{max}}(\bar{G}) = |V| - \Delta_{\text{min}}(G)$, the simplified version in the introduction follows as an immediate corollary. We now take a look at two corollaries of this bound: first to the structure of non-expanding sets on $3$-dimensional clique complexes, then for general $k$.
\begin{corollary}\label{cor:clique}
Let $G=(V,E)$ be a graph with $\Delta_{\text{min}} \geq \frac{5}{6}|V|$. Then any subset $S \subset K_{G,3}(3)$ of the $3$-clique complex with expansion at most 
\[
\Phi(S) \leq \frac{1}{3}-\delta
\]
must contain a constant fraction of the triangles touching some vertex $v \in G$:
\[
\exists v \in V: \underset{X_v}{\mathbb{E}}[1_S] \geq c\delta.
\]
\end{corollary}
\begin{proof}
Assume for the sake of contradiction that $f$ is $(\varepsilon,1)$-global. We will show it must be the case that $\varepsilon \geq \frac{\delta}{2}$. By \Cref{thm:clique} we have that
\[
\frac{1}{3} - \delta \geq \Phi(f) \geq \frac{1}{3}- \frac{1}{3}\mathbb{E}[f] - \frac{5}{4}\varepsilon
\]
and therefore that
\[
\varepsilon \geq \frac{4}{5}\delta - \frac{4}{15}\mathbb{E}[f].
\]
We can assume without loss of generality that $\mathbb{E}[f] \leq \frac{1}{2}\delta$ (otherwise the result is trivial by averaging), which in turn gives $\varepsilon \geq \frac{\delta}{2}$ as desired.
\end{proof}
We are not aware of any other method for showing such a characterization on low-dimensional clique complexes. It is further worth noting that while requiring min-degree at least $\frac{5}{6}|V|$ seems very restrictive (and indeed does imply the resulting complex is dense), clique complexes of much lower degree complexes are not even guaranteed to be connected (consider, e.g.\ two disjoint copies of the complete graph $K_{n/2}$, which has min-degree roughly $\frac{5}{6}|V|$).


For large $k$, the same proof strategy can be used to show that any non-expanding set has constant density in some $k-O(1)$ dimension link.
\begin{corollary}
There is a universal constant $c>0$ such that for any $i,k \in \mathbb{N}$, $0 < \delta < 1$, and graph $G=(V,E)$ of minimum degree at least $\Delta_{\text{min}} \geq \frac{2k-2}{2k-1}|V|$, any subset $S \subset X_{G,k}(k)$ with expansion at most
\[
\Phi(S) \leq \frac{1}{i}(1-\delta)
\]
has constant correlation with a link of co-dimension $i$:
\[
\exists \tau \in X(k-i): \underset{X_\tau}{\mathbb{E}}[1_S] \geq c\delta.
\]
\end{corollary}

While characterizing non-expanding functions on common dense complexes is interesting in its own right, much of the promise of high dimensional expanders comes from their ability to give sparse models for dense objects such as the complete complex or product spaces. Prior analysis of small set expansion on such objects was limited to \textit{algebraic} constructions of HDX \cite{lubotzky2005explicit,kaufman2018construction,dikstein2023new}. As an immediate application of our expansion framework, we prove the following global expansion theorem and characterization of non-expanding sets on combinatorial `product' HDX. We refer the reader to \cite[Definition 12]{golowich2021improved} for the exact definition of the product construction we'll denote by $G~\circled{h}~Y$ for a graph $G$ and simplicial complex $Y$, and only use the fact that $\gamma_{k-2-i} \leq \frac{1}{k-i}$ below.
\begin{corollary}\label{cor:prod}
Let $k \geq 3$, $G$ be any graph on $n$ vertices, $X=G~\circled{h}~\Delta_{4k}(2k)$, $0 < i < k$, and $f \in C_{k}$ be any $(\varepsilon,i)$-global boolean function.
The expansion of $f$ with respect to the lower walk is at least:
\[
\Phi(f) \geq \frac{(1-\mathbb{E}[f])i}{(k-1)(k-i)} - 4\varepsilon
\]
\end{corollary}
We omit the proof, which is immediate from the fact that the product satisfies $\max\{\gamma_{k-2},|\gamma_{k-2}^{(-)}|\} \leq \frac{1}{k}$ from \cite{golowich2021improved} and Trickle-down (\Cref{thm:trickling}). As for the case of clique-complexes, this leads to a novel characterization of non-expanding sets in low dimensions. For concreteness, we look at the setting of $k=3$.
\begin{corollary}\label{cor:prod2}
Let $G$ be any graph on $n$ vertices, and $X=G~\circled{h}~\Delta_{12}(6)$. There exists a universal constant $c >0$ such that any $S \subset X(3)$ with expansion at most
\[
\Phi(S) \leq \frac{1}{4} - \delta
\]
must contain a constant fraction of the triangles touching some vertex $v \in X(1)$:
\[
\exists v \in X(1): \underset{X_v}{\mathbb{E}}[1_S] \geq c\delta.
\]
\end{corollary}
\subsection{Small-Set Expansion of the Ramanujan Complexes}

The Ramanujan complexes are the seminal construction of high dimensional expanders \cite{lubotzky2005explicit}, but are \textit{one-sided} expanders and may be non-partite, failing the strong conditions required by prior work. Nevertheless, the Ramanujan complexes come with a variety of other substantial benefits not necessarily enjoyed by other constructions, leading to applications ranging from property and agreement testing \cite{kaufman2020local,kaufman2014high} to quantum codes \cite{EvraKZ20,KaufmanT21quantum} and Sum-of-Squares lower bounds \cite{dinur2020explicit}. It is natural therefore to ask whether these complexes also satisfy an analogous small-set expansion theorem to other constructions with stronger spectral guarantees.

We refer the reader to \cite{lubotzky2005explicit} for the rather involved details of these constructions, and state here only the salient points for our applications.
\begin{theorem}[The LSV-Complexes {\cite{lubotzky2005explicit}}]
    For every $\gamma>0$ and $d \in \mathbb{N}$, there exists an infinite family of complexes (the `LSV-complexes') $\{(X,\Pi)_n\}$ satisfying:
    \begin{enumerate}
        \item $(X,\Pi)_n$ is a $d$-dimensional, \textbf{non-partite} $\gamma$-one-sided HDX on $n$ vertices
        \item For all $v \in X(1)$: $(X_v,\Pi_v)_n$ is a \textbf{partite} $\gamma$-one-sided HDX on $O_{d,\gamma}(1)$ vertices.
    \end{enumerate}
\end{theorem}
Since such complexes are one-sided and non-partite, they do not admit any (known) theory of Fourier analysis, and as a result the tools and results developed in \cite{dikstein2018boolean,bafna2020high,bafna2021hypercontractivity,gur2021hypercontractivity} do not apply. However, since their \textit{links} are partite HDX, we can appeal to the following result of \cite{hopkins2025hypercontractivity} to show the complexes are $1$-locally SSE.
\begin{theorem}[SSE for Partite HDX {\cite{hopkins2025hypercontractivity}}]\label{thm:ram-sse}
Let $(X,\Pi)$ be a $d$-dimensional partite $\gamma$-one-sided local-spectral expander. If $f \in C_d$ is $(\varepsilon,i)$-strongly global, then the expansion of $f$ with respect to $T_\rho$ is at least:
\[
\Phi_{T_\rho}(f) \geq 1 - \rho^{i+1} - c_i\varepsilon^{1/2} - c_k\gamma
\]
where $c_j \leq 2^{O(j)}$.
\end{theorem}
We give the proof in \Cref{app:basis} for completeness. As an immediate corollary, we get a global SSE Theorem for the Ramanujan complexes:
\begin{corollary}\label{cor:LSV-expand}
Let $(X,\Pi)_n$ be an LSV-complex with $\gamma$-one-sided local-spectral expansion, and $f \in C_k$ an $(\varepsilon,i)$-global set of density at most $\varepsilon$. Then the expansion of $f$ is at least:
\[
\Phi_{T_\rho}(f) \geq 1 - \rho^{i} - c_i\varepsilon^{1/2} - c_k\gamma - \rho^k
\]
where each $c_i \leq 2^{O(i)}$.
\end{corollary}
\begin{proof}
     By \Cref{thm:ram-sse}, the LSV complexes $(X,\Pi)_n$ are 1-locally $\phi$-SSE for
    \[
    \phi(\varepsilon,i) = 1 - \rho^{i+1} - c_i\varepsilon^{1/2} - c_k\gamma
    \]
    The result is now immediate from \Cref{thm:local-to-global}.
\end{proof}
Equivalently, we can view this result via its contrapositive: any non-expanding set is local.
\begin{corollary}
For any $\rho \in [0,1]$, $k>\Omega(\rho^{-1})$, LSV-complex $(X,\Pi)_n$ with $\gamma \leq 2^{-\Omega(k)}$ and any $\delta>0$, there exist constants $r=r(\delta,\rho)$ and $s=s(\delta,\rho)$ such that for any boolean $f \in C_k$ with expansion $\bar{\Phi}_{T_{\rho}}(f) > \delta$, $f$ has density at least $s$ in some $r$-link:
    \[
    \exists \tau \in X(r): \underset{X_\tau}{\mathbb{E}}[f|_\tau] \geq s
    \]
    Moreover, we can take $r \leq O(\frac{\log \delta^{-1}}{\log \rho^{-1}})$, and $s \geq 2^{-O(r)}$.
\end{corollary}
\section*{Acknowledgements}
We thank Yotam Dikstein, Louis Golowich, Tali Kaufman, Shachar Lovett, and Anthony Ostuni for discussion and comments on an earlier draft of this work.
\bibliographystyle{unsrt}  
\bibliography{references} 
\appendix

\section{Fourier Analysis on HDX}\label{app:basis}
Most of the Fourier analytic results needed for \Cref{cor:noisy-garland2} follow without too much difficulty from existing techniques. However since the results are not explicitly stated in any work (and often require pulling from several sources), we give a contained version here. Our main goal is to prove the following result:
\begin{theorem}[\Cref{thm:Fourier-Analysis} Restated]
    Let $(X,\Pi)$ be a $k$-dimensional $\gamma$-two-sided or partite one-sided HDX. Then there exists a decomposition $f=\sum\limits_{i=0}^k f_i$ such that:
    \[
    \bar{\Phi}_{N^{k-i}_k}(f) = \sum\limits_{\ell=0}^i \frac{{k - \ell \choose i - \ell}}{{k \choose i}} \frac{\langle f_\ell, f_\ell \rangle}{\langle f, f \rangle} \pm c_k\gamma
    \]
    where $c_k \leq 2^{O(k)}$. Similarly:    \[
    \bar{\Phi}_{T_\rho^k}(f) = \sum\limits_{\ell=0}^k \rho^\ell \frac{\langle f_\ell, f_\ell \rangle}{\langle f, f \rangle} \pm c_k\gamma
    \]
\end{theorem}
The two-sided variant is actually proven explicitly in \cite{bafna2020high,bafna2021hypercontractivity}, so we focus on the partite one-sided case which is studied implicitly in \cite{gur2021hypercontractivity}. More formally, the authors study a notion of expansion on partite complexes they call `$\gamma$-products.' To define these, we need to introduce a new type of random walk between colors on a partite complex, originally introduced by Dikstein and Dinur \cite{dikstein2019agreement} in the context of agreement testing.
\begin{definition}[Colored Swap-Walks {\cite{dikstein2019agreement}}]
    Given a partite simplicial complex $(X,\Pi)$ and $i,j \in [d]$, the colored swap-walk $M^{i,j}$ walks from $X^i$ to $X^j$ and has bipartite adjacency operator:
    \[
    M^{i,j}(v,w) = \frac{\pi_{w,1}(v)}{\sum\limits_{z \in X^i} \pi_{w,1}(z)}.
    \]
\end{definition}
On a product space, every $M^{i,j}$ is a complete bipartite graph (since colors are completely independent). Gur, Lifshitz, and Liu introduced the notion of a $\gamma$-product to relax this constraint to simply requiring these operators expand.
\begin{definition}[$\gamma$-Product {\cite{gur2021hypercontractivity}}]
    A partite complex $(X,\Pi)$ is a $\gamma$-product if every link $X_\tau$ of co-dimension at least 2, every colored swap-walk is a $\gamma$-one-sided spectral expander:
    \[
    \forall i,j: \lambda_2(M_\tau^{i,j}) \leq \gamma,
    \]
\end{definition}
It is not hard to show that $\gamma$-products are \textit{equivalent} to one-sided partite HDX up to a factor in dimension. Indeed the `hard' direction (HDX $\to$ $\gamma$-product) was already shown in prior work of Dikstein and Dinur \cite{dikstein2019agreement}, while the `easy' direction follows from Oppenheim's Trickling-Down Theorem.

\begin{theorem}[$\gamma$-Product $\iff$ Partite HDX]
    Let $(X,\Pi)$ be a simplicial complex and $\gamma < \frac{1}{d-1}$. Then:
    \begin{enumerate}
        \item If $(X,\Pi)$ is a $\gamma$-one-sided HDX, then it is a $\frac{\gamma}{1-(d-2)\gamma}$-product
        \item If $(X,\Pi)$ is a $\gamma$-product, it is a $\frac{\gamma}{1-(d-2)\gamma}$-one-sided HDX
    \end{enumerate}
\end{theorem}
\begin{proof}
    The first result is \cite[Corollary 7.6]{dikstein2019agreement}, and follows from a partite variant of the Trickling-Down Theorem. To prove the second, observe that 
    \begin{enumerate}
        \item For $\tau \in X(d-2)$: $A_\tau=M_\tau^{1,2}$,
        \item For any $i \leq d-2$ and $\tau \in X(i)$: $A_\tau$ is connected.\footnote{if $A_\tau$ is disconnected, there exist disconnected vertices $v_i$ and $v_j$ of colors $i$ and $j$. By construction, $v_i$ and $v_j$ are then also disconnected in $M_\tau^{i,j}$, which violates the expansion assumption.}
    \end{enumerate}
    Since we are promised $\lambda_2(M_\tau^{1,2}) \leq \gamma < \frac{1}{d-1}$ and all links are connected, Trickling-Down (\Cref{thm:trickling}) implies the result.
\end{proof}
GLL show that $\gamma$-products, and therefore partite one-sided HDX, have an approximate Fourier basis. Their result is based off a classical basis for product spaces called the \textit{Efron-Stein Decomposition}. To state this, it will first be useful to define a family of partite averaging operators 
\begin{definition}[Partite Averaging Operators]
    Given a partite simplicial complex $(X,\Pi)$, the partite averaging operators $\{E_T\}_{T \subseteq [d]}: C(d) \to C(d)$ average over elements in $[n]\setminus T$:
    \[
    E_Tf(x) = \underset{X_{x_T}}{\mathbb{E}}[f|_{x_T}],
    \]
    i.e.\ $E_Tf$ is the conditional expectation of $f$ given that coordinates $T$ are fixed to $x_T$.
\end{definition}
The Efron-Stein decomposition divides $f$ into components $f^{=S}$ that make up the contribution to $f$ from each coordinate set, which essentially amounts to applying inclusion-exclusion to $E_Sf$.
\begin{definition}
 Given a partite simplicial complex $(X,\Pi)$ and $f$, define
 \[
 f^{=S} = \sum\limits_{T \subseteq S}(-1)^{|S \setminus T|}E_Tf
 \]
 We denote the sum of components at a given level as $f_i \coloneqq \sum\limits_{|S|=i}f^{=S}$
\end{definition}
It is not hard to check that $\{f^{=S}\}$ decompose $f$, that is.
\[
f=\sum\limits_{S \subseteq [d]}f^{=S}
\]
GLL proved that the Efron-Stein decomposition is an approximate Fourier basis for the partite averaging operators in the following sense.
\begin{theorem}[Efron-Stein Decomposition on HDX]\label{thm:hdx-ES}
Let $(X,\Pi)$ be a partite $\gamma$-one-sided HDX. Then the Efron-Stein Decomposition is approximately orthogonal:
\[
\forall S,T: \langle f^{=S},f^{=T}\rangle \leq 2^{O(|S|+|T|)}\gamma\norm{f}_2^2,
\]
and approximately an eigenbasis for every $E_T$:
\begin{enumerate}
    \item if $S \subseteq T: E_Tf^{=S}=f^{=S}$
    \item If $S \nsubseteq T: \norm{E_Tf^{=S}}_2 \leq \sqrt{|S||T|}2^{|S|}\gamma \norm{f}$
\end{enumerate}
\end{theorem}
high order random walks on partite complexes can typically be written as convex combinations of the partite averaging operators. With this in mind, the proof of \Cref{thm:Fourier-Analysis} follows easily from expanding $f$ into the Efron-Stein basis.
\begin{proof}[Proof of \Cref{thm:Fourier-Analysis}]
    Observe that the lower walks can be written as the following convex combination of the colored averaging operators:
    \[
    N_k^{k-i} = \frac{1}{{k \choose i}}\sum\limits_{|T|=i}E_T
    \]
    It then follows from \Cref{thm:hdx-ES} that Efron-Stein is also an approximate eigenbasis for $N_k^i$:
    \begin{align*}
        N_k^{k-i}f^{=S} &=\frac{1}{{k \choose i}}\sum\limits_{|T|=i}E_Tf^{=S}\\
        &= \frac{1}{{k \choose i}}\sum\limits_{\underset{T \supseteq S}{|T|=i}}E_Tf^{=S} + \frac{1}{{k \choose i}}\sum\limits_{\underset{T \nsupseteq S}{|T|=i}}E_Tf^{=S}\\
        &= \frac{{k-|S| \choose i-|S|}}{{k \choose i}}f^{=S} + \vec{err}
    \end{align*}
where $\norm{\vec{err}}_2 \leq 2^{O(k)}\gamma$. Note that the lefthand term is really $0$ for $|S|>i$ (which we've denoted as a negative binomial coefficient for convenience).
    
The result for $N_k^{k-i}$ then follows from expanding out $\Phi(f)$ in this basis and applying approximate orthogonality:
\begin{align*}
    \Phi(f) &= 1 - \frac{1}{\langle f,f \rangle}\langle f, \widehat{N}_k^if \rangle\\
    &=1 - \frac{1}{\langle f,f \rangle}\sum\limits_{\ell=0}^k\langle f, \widehat{N}_k^if_\ell \rangle\\
       &=1 - \frac{1}{\langle f,f \rangle}\sum\limits_{\ell=0}^{i} \frac{{k-|S| \choose i-|S|}}{{k \choose i}}\langle f,f_\ell \rangle + c_1\gamma\\
       &=1 - \frac{1}{\langle f,f \rangle}\sum\limits_{\ell=0}^{i} \frac{{k-|S| \choose i-|S|}}{{k \choose i}}\langle f, f_\ell \rangle + c_2\gamma
\end{align*}
where $c_1,c_2 \leq 2^{O(k)}$. The result for $T_\rho$ follows similarly, in fact as GLL (\cite[Claim 8.3]{gur2021hypercontractivity}) note the expansion decomposition in this case is exact:
\[
\Phi_{T_\rho}(f) = \frac{1}{\langle f,f \rangle}\sum\limits_{i=0}^k \rho^{i}\langle f, f_i, \rangle
\]
so approximate orthogonality again implies the desired result.
\end{proof}
\section{Gotlib-Kaufman for the Lower Walk}\label{app:GK}
\begin{theorem}\label{thm:decomp-app}
For any simplicial complex $(X,\Pi)$, $0 < k \leq d$, and $f \in C_{k}$, we have:
\[
\langle N_k^1 f,f \rangle \leq \mathbb{E}[f]^2 + \sum\limits_{i=1}^{k-1}\left(1-\frac{1}{k-i+1}\prod\limits_{j=i-1}^{k-2}(1-\gamma_j) \right) \langle f,f_i \rangle.
\]
\end{theorem}
Following GK \cite{gotlib2022fine}, the proof revolves around reducing to level-$1$ by (inductively) localizing the decomposition. The crucial observation is that in each step, it is enough to analyze the contribution coming from the constant part of level 1, which is handled by the following lemma.
\begin{lemma}[{\cite[Lemma 7.8, Theorem 7.9]{gotlib2022fine}}]\label{lemma:advantage}
Let $(X,\Pi)$ be a simplicial complex and $g \in C_k$ any function orthogonal to $\vec{1}$. Then the parallel part of $g$ under restriction is small in expectation:
\[
\underset{v \sim \pi_1}{\mathbb{E}}\left[\norm{(N|_v)^{k-1}_{k-1}g|_v}_2^2\right] \leq \left(1-\frac{k-1}{k}(1-\gamma_0)\right)\norm{g}_2^2
\]
\end{lemma}
\begin{proof}
First, observe that since $(D|_v)^{k-1}_0g|_v$ is constant, we have:
\[
\norm{(N|_v)^{k-1}_{k-1}g|_v}_2^2=D^k_1g(v)^2
\]
and therefore that
\[
\underset{v \sim \pi_1}{\mathbb{E}}\left[\norm{(N|_v)^{k-1}_{k-1}g|_v}_2^2\right] = \langle D^k_1 g, D^k_1 g \rangle = \langle N_k^{k-1} g,g \rangle
\]
by adjointness of $D$ and $U$. Since $g$ is orthogonal to $\vec{1}$, it is enough to upper bound $\lambda_2(N_k^{k-1})$. It will be convenient to instead work with the upper walk $\widehat{N}_1^{k-1}$, which has the same non-zero spectrum.

In particular, it is well known that the non-lazy component of $\widehat{N}_1^{k-1}$ is simply $M_1^+$ (see e.g.\ \cite{alev2019approximating,gotlib2022fine}), so we can write:
\[
\lambda_2(\widehat{N}_1^{k-1}) = \lambda_2\left(\frac{1}{k}I + \frac{k-1}{k}M_1^+\right) \leq \frac{1}{k} + \frac{k-1}{k}\gamma_0 = 1-\frac{k-1}{k}(1-\gamma_0)
\]
as desired.
\end{proof}
The proof of the main theorem then follows from combining this fact with a basic induction on $k$, performed by successive restrictions.
\begin{proof}[Proof of \Cref{thm:decomp-app}]
The proof follows largely the same strategy as \cite[Theorems 5.8,7.9]{gotlib2022fine} where the only real difference lies in the base case. The proof is by induction on $k$. Note that since $f_0=\mathbb{E}[f]\vec{1}$, and the remaining $f_i$ are orthogonal, it is enough to consider the case where $f_0=0$.

Because the KO-decomposition does not behave well under restriction, as in \cite{gotlib2022fine} we will instead prove the more general statement for any sum $\sum\limits_{i=1}^k f_i$ satisfying the weaker constraint $f_i \in Ker(D^k_{i-1})$:
\begin{equation}
\left\langle N_k^1 \sum\limits_{i=1}^k f_i,\sum\limits_{i=1}^k f_i \right\rangle \leq \sum\limits_{i=1}^{k-1}\left(1-\frac{1}{k-i+1}\prod\limits_{j=i-1}^{k-2}(1-\gamma_j) \right) \langle f_i,f_i \rangle + \sum_{i \neq j} c_{i,j}\langle f_i,f_j \rangle\\
\end{equation}
Since the KO-decomposition is orthogonal and satisfies $f_i \in Ker(D^k_{i-1})$, this is sufficient to prove the result. For notational convenience, we continue to write $f= \sum\limits_{i=1}^k f_i$ throughout.
With this in mind, we start with the base case of graphs. Notice that by adjointness of $U$ and $D$ we have:
\begin{align*}
\langle N_2^1f, f \rangle = \langle U_{1}D_2(f_1+f_2),f_1+f_2 \rangle &= \langle D_2f_1,D_2f_1 \rangle
\end{align*}
since $f_2 \in \text{Ker}(D_2)$ by definition. The result now follows from \cite[Lemma 7.8]{gotlib2022fine}, who shows that
\[
\langle D_2f_1,D_2f_1 \rangle \leq \left(1-\frac{1}{2}(1-\gamma_0)\right)\norm{f_1}^2_2
\]
as desired (this also follows from the proof of \Cref{lemma:advantage} above).

The inductive step follows essentially the same argument as Gotlib-Kaufman (replacing relevant parameters throughout), but we include the proof for completeness. Because restriction is linear and respected by the lower walk, we have
\begin{align*}
\langle N_k^1f, f \rangle &= \left\langle N_k^1\sum\limits_{i=1}^k f_i, \sum\limits_{i=1}^k f_i \right\rangle\\
&= \underset{v \sim \pi_1}{\mathbb{E}}\left[ \left\langle (N|_v)_{k-1}^1\sum\limits_{i=1}^k f_i|_v, \sum\limits_{i=1}^k f_i|_v \right\rangle \right].
\end{align*}
Now that we have reduced to dimension $k-1$ we'd hope to apply the inductive hypothesis, but this requires a sum of the form $\sum_{i=1}^{k-1}h_i$ for $h_i \in \text{Ker}((D|_v)^{k-1}_{i-1})$. Gotlib and Kaufman observed that the only part of the previous expression that fails to achieve this form is the constant part of $f|_v$. In particular, letting
\begin{align*}
    h^v_i &\coloneqq (f_{i+1})|_v & (1 < i < k)\\
    h^v_i &\coloneqq (f_2)|_v + \left(1-(N|_v)_{k-1}^{k-1}\right)(f_1)|_v & (i = 1),
\end{align*}
we have $h^v_i \in \text{Ker}((D|_v)^{k-1}_{i-1})$ and can write:
\begin{align*}
\langle N_k^{1}f, f \rangle =& \underset{v \sim \pi_1}{\mathbb{E}}\left[\left\langle (N|_{v})_{k-1}^{1}\sum\limits_{i=1}^{k-1} h^v_i, \sum\limits_{i=1}^{k-1} h^v_i\right\rangle \right]\\
&+ \underset{v \sim \pi_1}{\mathbb{E}}\left[\left\langle (N|_{v})_{k-1}^{1}(N|_v)_{k-1}^{k-1}(f_1)|_v, (N|_v)_{k-1}^{k-1}(f_1)|_v \right\rangle \right].
\end{align*}
The first term can now be analyzed by the inductive hypothesis:
\begin{align*}
   \left\langle (N|_{v})_{k-1}^{1}\sum\limits_{i=1}^{k-1} h^v_i, \sum\limits_{i=1}^{k-1} h^v_i\right\rangle &\leq \sum\limits_{i=1}^{k-2}\left(1-\frac{1}{k-i}\prod\limits_{j=i}^{k-2}(1-\gamma_j) \right) \langle h^v_i,h^v_i \rangle + \sum_{i \neq j} c_{i,j}\langle h^v_i,h^v_j \rangle\\
    &=\sum\limits_{i=2}^{k-1}\left(1-\frac{1}{k-i+1}\prod\limits_{j=i-1}^{k-2}(1-\gamma_j) \right) \langle (f_{i})|_v,(f_{i})|_v \rangle\\
    &+ \left(1-\frac{1}{k-1}\prod\limits_{j=1}^{k-2}(1-\gamma_j) \right)\norm{\left(1-(N|_v)_{k-1}^{k-1}\right)(f_1)|_v}_2^2\\ 
    &+ \sum_{i \neq j} c'_{i,j}\langle (f_i)|_v,(f_j)_v \rangle
\end{align*}
where we have re-applied the definition of $h^v_i$ and collected mixed terms as in \cite[Lemma 5.10]{gotlib2022fine}. Recall that by Garland's method for any two functions $g,g'$ we have
\[
\underset{v \sim \pi_1}{\mathbb{E}}[\langle g|_v,g'|_v \rangle] = \langle g,g' \rangle,
\]
and thus altogether that
\begin{align*}
\langle N_k^{1}f, f \rangle &\leq \sum\limits_{i=2}^{k-1}\left(1-\frac{1}{k-i+1}\prod\limits_{j=i-1}^{k-2}(1-\gamma_j) \right) \langle f_i,f_i \rangle\\
    &+ \underset{v \sim \pi_1}{\mathbb{E}}\left[\left(1-\frac{1}{k-1}\prod\limits_{j=1}^{k-2}(1-\gamma_j) \right)\norm{\left(1-(N|_v)_{k-1}^{k-1}\right)(f_1)|_v}_2^2 + \norm{(N|_v)_{k-1}^{k-1}(f_1)|_v}_2^2 \right]\\ 
    &+ \sum_{i \neq j} c'_{i,j}\langle f_i,f_j \rangle
\end{align*}
where we have used the fact that $N_k^1(N|_v)_{k-1}^{k-1}(f_1)|_v=(N|_v)_{k-1}^{k-1}(f_1)|_v$ since the latter is constant. For notational simplicity let
\[
\lambda_{i,k} \coloneqq \left(1-\frac{1}{k-i+1}\prod\limits_{j=i-1}^{k-2}(1-\gamma_j) \right).
\]
It is therefore left to prove the following claim:
\begin{claim}[{\cite[Theorems 5.8,7.9]{gotlib2022fine}}]
\begin{equation}
    \underset{v \sim \pi_1}{\mathbb{E}}\left[\lambda_{k,2}\norm{\left(1-(N|_v)_{k-1}^{k-1}\right)(f_1)|_v}_2^2 + \norm{(N|_v)_{k-1}^{k-1}(f_1)|_v}_2^2 \right] \leq \lambda_{k,1}\norm{f_1}_2^2
\end{equation}
\end{claim}
This follows from combining several more general claims in Gotlib-Kaufman, but we'll repeat the direct version here for completeness. GK observe the lefthand side can be re-written as
\begin{align*}
    &\underset{v \sim \pi_1}{\mathbb{E}}\left[\lambda_{k,2}\norm{\left(1-(N|_v)_{k-1}^{k-1}\right)(f_1)|_v}_2^2 + \lambda_{k,2}\norm{(N|_v)_{k-1}^{k-1}(f_1)|_v}_2^2 + (1-\lambda_{k,2})\norm{(N|_v)_{k-1}^{k-1}(f_1)|_v}_2^2 \right] \\
    =&\lambda_{k,2}\underset{v \sim \pi_1}{\mathbb{E}}[\norm{(f_1)|_v}_2^2] + (1-\lambda_{k,2})\underset{v \sim \pi_1}{\mathbb{E}}\left[\norm{(N|_v)_{k-1}^{k-1}(f_1)|_v}_2^2 \right]\\
    =&\lambda_{k,2}\norm{f_1}_2^2 + (1-\lambda_{k,2})\underset{v \sim \pi_1}{\mathbb{E}}\left[\norm{(N|_v)_{k-1}^{k-1}(f_1)|_v}_2^2 \right]
\end{align*}
With this in mind, we are finally ready to apply \Cref{lemma:advantage}
\[
\underset{v \sim \pi_1}{\mathbb{E}}\left[\norm{(N|_v)_{k-1}^{k-1}(f_1)|_v}_2^2 \right] \leq \left(1 - \frac{k-1}{k}(1-\gamma_0)\right)\norm{f_1}_2^2
\]
which completes the proof since
\begin{align*}
    \lambda_{k,2} + (1-\lambda_{k,2})\left(1 - \frac{k}{k+1}(1-\gamma_0)\right) &= 1 - (1-\lambda_{k,2})\left(1 - \frac{k}{k+1}(1-\gamma_0)\right)\\
    &=1- \frac{1}{k}\prod\limits_{j=0}^{k-2}(1-\gamma_j)\\
    &=\lambda_{k,1}
\end{align*}
as desired.
\end{proof}

\section{Garland's Lemma}
We prove the lower-walk variant of Garland's Lemma from \Cref{lemma:garland}.
\begin{lemma}[Garland's Lemma for Lower Walks]
Let $(X,\Pi)$ be a weighted $d$-dimensional simplicial complex. Then for any $i<k<d$, $f \in C_k$, $j \leq k-i$ and $\tau \in X(j)$:
\[
\langle f, N_k^i f \rangle = \mathbb{E}_{\tau \sim \pi_j}\left[\langle f|_\tau, N_{k-j}^i|_\tau f|_\tau \right\rangle]
\]
where $f|_\tau: X_\tau(k-i) \to \mathbb{R}$ is given by $f_\tau(\sigma) = f(\tau \cup \sigma)$, and $N_{k-j}^i|_\tau$ is the lower walk on the link $X_\tau$.
\end{lemma}
\begin{proof}
Since $D$ and $U$ are adjoint, we can write:
\begin{align*}
\mathbb{E}_{\tau \sim \pi_j}\left[\langle f|_\tau, N_{k-j}^i|_\tau f|_\tau \rangle \right] &= 
\mathbb{E}_{\tau \sim \pi_j}\left[\langle D^{k-j}_{k-i-j}|_\tau f|_\tau, D^{k-j}_{k-i-j}|_\tau f|_\tau \rangle \right]\\
&= \sum\limits_{\tau \in X(j)} \pi_j(\tau) \sum\limits_{\sigma \in X_\tau(k-i-j)} \pi_{\tau,k-i-j}(\sigma)(D^{k-j}_{k-i-j}|_\tau f|_\tau (\sigma))^2\\
&= \frac{1}{{k-i \choose j}}\sum\limits_{\tau \in X(j)} \sum\limits_{\sigma \in X_\tau(k-i-j)} \pi_{k-i}(\tau \cup \sigma)(D^{k-j}_{k-i-j}|_\tau f|_\tau (\sigma))^2.
\end{align*}
The key is now to observe that $D^{k-j}_{k-i-j}|_\tau f|_\tau (\sigma)$, the expected value of $f|_\tau$ over the link of $\sigma$ in $X_\tau$ is equivalent to the expectation of $f|_{\tau \cup \sigma}$ over the link $X_{\tau \cup \sigma}$, which is just $D^k_{k-i}f(\tau \cup \sigma)$. Therefore we have:
\begin{align*}
    \frac{1}{{k-i \choose j}}\sum\limits_{\tau \in X(j)} \sum\limits_{\sigma \in X_\tau(k-i-j)} \pi_{k-i}(\tau \cup \sigma)(D^{k-j}_{k-i-j}|_\tau f|_\tau (\sigma))^2 &= \frac{1}{{k-i \choose j}}\sum\limits_{\tau \in X(j)} \sum\limits_{\sigma \in X_\tau(k-i-j)} \pi_{k-i}(\tau \cup \sigma)(D^k_{k-i} f (\tau \cup \sigma))^2\\
    &= \sum\limits_{T \in X(k-i)}\pi_{k-i}(T) (D^k_{k-i}f(T))^2\\
    &= \langle f,N^i_k f \rangle,
\end{align*}
where the second to last step follows from noting that the inner term depends only on the union $\tau \cup \sigma \in X(k-i)$, each of which is hit ${k-i \choose j}$ times in the sum. 
\end{proof}
\section{Proof of \Cref{claim:binomial}}
\begin{claim}[\Cref{claim:binomial} Restated]
For all $k \in \mathbb{N}$, $j \leq k$, and $\ell \leq k-j$:
    \[
    \sum_{i=\ell}^{k-j}{k \choose i+j}\frac{{k-j-\ell \choose i-\ell}}{{k-j \choose i}}\rho^{i+j}(1-\rho)^{k-j-i} \leq \rho^\ell
    \]
\end{claim}
\begin{proof}
We first re-write the sum to be $0$-indexed for simplicity:
    \begin{align*}
    \sum_{i=\ell}^{k-j}{k \choose i+j}\frac{{k-j-\ell \choose i-\ell}}{{k-j \choose i}}\rho^{i+j}(1-\rho)^{k-j-i}= \sum_{i=0}^{k-j-\ell}\frac{{k \choose i+j+\ell}{k-j-\ell \choose i}}{{k-j \choose i+\ell}}\rho^{i+j+\ell}(1-\rho)^{k-j-\ell-i}
    \end{align*}
The key is to observe that by standard counting arguments:
\begin{equation}\label{eq:binomial}
\frac{{k \choose i+j+\ell}{k-j-\ell \choose i}}{{k-j \choose i+\ell}} \leq {k-\ell \choose i+j},
\end{equation}
since by the Binomial Theorem we then have
\begin{align*}
    \sum_{i=0}^{k-j-\ell}{k \choose i+j+\ell}\frac{{k-j-\ell \choose i}}{{k-j \choose i+\ell}}\rho^{i+j+\ell}(1-\rho)^{k-j-\ell-i} &\leq \rho^{\ell}\sum_{i=0}^{k-j-\ell}{k-\ell \choose i+j}\rho^{i+j}(1-\rho)^{k-j-\ell-i}\\
    &\leq \rho^\ell\sum_{z=0}^{k-\ell}{k-\ell \choose z} \rho^{z}(1-\rho)^{k-\ell-z}\\
    &= \rho^\ell
\end{align*}
as desired. Proving \Cref{eq:binomial} is tedious but elementary. Standard binomial manipulations give
\begin{align*}
    \frac{{k \choose i+j+\ell}{k-j-\ell \choose i}}{{k-j \choose i+\ell}{k-\ell \choose i+j}}  &= \frac{{k \choose j}}{{k-\ell \choose j}}\cdot \frac{{i+j \choose j}}{{i+j+\ell \choose j}},
\end{align*}
and the result follows from observing the RHS is increasing in $i$, and is exactly $1$ when $i=k-j-\ell$.
\end{proof}
\end{document}